\documentclass[11pt]{amsart}
\usepackage{amsmath,amssymb,amsfonts,amscd,a4}
\usepackage[cmtip, all]{xy}
\usepackage{pstricks}
\usepackage{color}
\usepackage{ulem}  
\usepackage{color} 
\usepackage{verbatim} %%% Added by Shiho 

\theoremstyle{plain}
\newtheorem{thm}{Theorem}[section]
\newtheorem{theorem}[thm]{Theorem}

\newtheorem{prop}[thm]{Proposition}
\theoremstyle{definition}
\newtheorem{remark}[thm]{Remark}

\numberwithin{equation}{section}
\newcommand{\ml}[2]{\begin{multline}\label{#1}#2 \end{multline}}
\newcommand{\ga}[2]{\begin{gather}\label{#1}#2 \end{gather}}

\newcommand{\surj}{\twoheadrightarrow}

\newcommand{\sD}{{\mathcal D}}
\newcommand{\sE}{{\mathcal E}}

\newcommand{\sH}{{\mathcal H}}
\newcommand{\sI}{{\mathcal I}}

\newcommand{\sK}{{\mathcal K}}

\newcommand{\sO}{{\mathcal O}}

\newcommand{\sS}{{\mathcal S}}
\newcommand{\sT}{{\mathcal T}}

\newcommand{\sX}{{\mathcal X}}

\newcommand{\sZ}{{\mathcal Z}}
\newcommand{\C}{{\mathbb C}}

\newcommand{\F}{{\mathbb F}}

\newcommand{\N}{{\mathbb N}}
\renewcommand{\P}{{\mathbb P}}
\newcommand{\Q}{{\mathbb Q}}

\newcommand{\Z}{{\mathbb Z}}

\title [Chern classes of crystals]{Chern classes of crystals}
\author{H\'el\`ene Esnault and Atsushi Shiho} 
\address{Freie Universit\"at Berlin, Arnimallee 3, 14195, Berlin,  Germany}
\email{esnault@math.fu-berlin.de}
\address{  Graduate School of Mathematical Sciences, 
the University of Tokyo, 3-8-1 Komaba, Meguro-ku, Tokyo 153-8914, Japan }
\email{ shiho@ms.u-tokyo.ac.jp}
\thanks{The first  author is supported by  the Einstein program. The second author is
 partly supported by JSPS 
Grant-in-Aid for Scientific Research (C) 25400008 and  
Grant-in-Aid for Scientific Research (A) 15H02048.}

\date{\today} %% changed by Shiho  

\begin{document}

\begin{abstract}  The crystalline Chern classes of the value of a locally free crystal vanish
on a smooth variety defined over a perfect field. Out of this we conclude new cases of de Jong's conjecture relating the geometric \'etale fundamental group of a smooth projective variety defined over a perfect field and the triviality of its category of isocrystals. We also discuss the case of the Gau{\ss}-Manin convergent isocrystal. 
\end{abstract}
\maketitle

\section{Introduction}\label{intro}
On  a smooth algebraic variety  $X$  defined over the field $\C$ of complex numbers, a vector bundle $E$ endowed with an integrable connection $\nabla: E \to \Omega^1_{X}\otimes_{\sO_X} E$ has vanishing Chern classes $c^{dR}_i(E)$ in de Rham cohomology $H^{2i}_{dR}(X/\C)$.

\medskip

The standard way to see this is applying Chern-Weil theory: the successive traces of the iterates of the curvature  $\nabla^2\in {\rm Hom}(E, \Omega^2_{X} \otimes E)$ are identified with the Newton classes $N_i(E)$, and the $\Q$-vector spaces spanned by $c^{dR}_i(E), 1\le i\le n$  and the $N_i(E), 1\le i\le n$ are the same in $H^i_{dR}(X)$ \cite{Che52}.  In particular,  the method loses torsion information, and for example, a torsion class in  integral $\ell$-adic cohomology $H^2(X_{\bar k}, \Z_\ell(1))$ for some $\ell$ is  the first $\ell$-adic Chern class of some line bundle which carries an integrable connection.

\medskip

If $(X,E)$ is defined over a field $k$ of characteristic $0$ and $E$ admits an integrable connection after base changing to $\C$ for a complex embedding $k\hookrightarrow \C$, one still has vanishing $0=c^{dR}_i(E)\in H^{2i}_{dR}(X/k)$ for $i\ge 1$
because Chern classes in de Rham cohomology are functorial and de Rham cohomology satisfies base change property as ultimately coherent cohomology does. 

\medskip

The first purpose of this article is to present a similar vanishing statement where $k$ is now a perfect field of characteristic $p>0$, 
$E$ is replaced by the value $E_X$ on $X$ of a $p$-torsion free crystal $E$, and de Rham cohomology is replaced by crystalline cohomology.  Let $W=W(k)$ be the ring of Witt vectors on $k$.   We denote by $H^i_{\rm crys}(X/W)$ 
the integral crystalline cohomology of $X$ 
% by $H^i(X/K)=H^i(X/W)\otimes_W K$ the associated $K$-vector space 
and by $c_i^{\rm crys}(E_X)$ 
%(resp. $c_i^{\rm crys}(E_X)_{\Q}$) 
the crystalline Chern classes of $E_X$ in $H^{2i}_{\rm crys}(X/W)$. 
%(resp. $H^{2i}(X/K)$). 

\medskip
Recall (see  \cite[Section~1]{ES15} for an overview of the concepts)  that a crystal is a sheaf of $\sO_{X/W}$-modules of finite presentation on the crystalline site of $X/W$, such that the transition maps are isomorphisms, that crystals build a $W$-linear category ${\rm Crys}(X/W)$, of which the $\Q$-linearization ${\rm Crys}(X/W)\xrightarrow{\Q\otimes }{\rm Crys}(X/W)_{\Q}$ is the category of isocrystals. Then ${\rm Crys}(X/W)_{\Q}$ is a Tannakian category over $K$, the fraction field of $W$. Any $\sE \in {\rm Crys}(X/W)_{\Q}$ is of the shape $\Q\otimes E$ where $E$ is a lattice, that is a $p$-torsion free crystal.  Yet it is an open question  whether  one can choose $E$ to be locally free.  
\begin{thm} \label{thm1:main}
Let $X$ be a smooth variety defined over a perfect field $k$ of characteristic $p>0$. 
If $E\in {\rm Crys}(X/W)$ is a locally free crystal, 
%\begin{itemize}
%\item[i)] If  $X$ is projective, one has $c_i^{\rm crys}(E_X)_{\Q}=0$ for $i\ge 1$. 
%\item[ii)] 
then 
$c_i^{\rm crys}(E_X)=0$ for $i\ge 1$. 
%\end{itemize}
\end{thm}
It is proved in \cite[Prop.~3.1]{ES15} that given an isocrystal $\Q\otimes E \in  {\rm Crys}(X/W)_{\Q}$, $c_i^{\rm crys}(E_X)$ does not depend on the choice of the lattice $E$. In particular, if this was true that any isocrystal carries a locally free lattice, then
% Theorem~\ref{thm1:main} i) would be true integrally. 
Theorem~\ref{thm1:main} would imply that $c_i^{\rm crys}(E_X)=0$ for $i\ge 1$ for any $p$-torsion free crystal $E\in {\rm Crys}(X/W)$.

The proofs of Theorem \ref{thm1:main} imply also the following variant 
for Chern classes in torsion crystalline cohomology: Let $W_n := W/p^nW$. 
Then, if $X$ is as in Theorem \ref{thm1:main} and
if $E$ is a locally free crystal on $X/W_n$, 
then $c_i^{\rm crys}(E_X)$ is zero in the torsion crystalline cohomology group 
$H^{2i}_{\rm crys}(X/W_n)$ for $i\ge 1$. See Remarks \ref{rem:torsion1} and \ref{rem:torsion2}. 

%We nonetheless remark 
%\begin{prop} \label{prop:locfree}
%Let $X$ be smooth projective over a perfect field $k$, which lifts to $W_2(k)$. Let $\sE$ be an isocrystal on $X$ of rank $\le p$ such that the irreducible constituents of  its Jordan-H\"older filtration have rank $\le p$.  Then $\sE$ admits a locally free lattice.

%\end{prop}
%Indeed, by \cite[Thm.4.4. 5)]{ES15}, the proposition is true when $\sE $ is convergent, and the only missing information to enhance it to the general case of an isocrystal is the vanishing of the numerical Chern classes, which follows from Theorem~\ref{thm1:main} i).
%We shall not repeat the proof in the article. 

\medskip

Recall that the Frobenius acts on ${\rm Crys}(X/W)$ and ${\rm Crys}(X/W)_{\Q}$. Locally ${\rm Crys}(X/W)$ is equivalent to the category of quasi-nilpotent integrable connections on a formal lift of $X$ over $W$, and the action is just given by the Frobenius pull-back of a connection. The category ${\rm Conv}(X/K)$ of convergent isocrystals is the largest full subcategory of ${\rm Crys}(X/W)_{\Q}$ which is stabilized by the Frobenius action.  It is proved in \cite[Prop.~3.1]{ES15} that $c_i^{\rm crys}(E_X)=0$ for any lattice $E$ of a convergent isocrystal, regardless of the existence of a locally free lattice.
   
\medskip

For a natural number $r$, set $N(r)$ to be the maximum of the lower common multiples of $a$ and $b$ for all choices $a, b\ge 1, a+b \le r$. 
As in \cite{ES15}, Theorem~\ref{thm1:main} enables one to prove the following case of de Jong's conjecture (\cite[Conj.~2.1]{ES15}).

\begin{thm}\label{thm:dejonglocfree}
Let $X$ be a smooth projective variety over a perfect field $k$ of characteristic $p>0$. 
If the \'etale fundamental group of $X \otimes_k \overline{k}$ is trivial 
and the maximal Mumford slope of the sheaf of $1$-forms is bounded above by $N(r)^{-1}$, then any isocrystal $\sE$ which is an iterated extension of 
irreducible isocrystals of rank $\le r$ having locally free lattices is isomorphic to $\sO_{X/K}^{\oplus {{\rm rank}( \sE)}}$. 
% {\red I think that each irreducible constituent should have a locally free lattice.} 
% {\blue I do not understand, it is what is written, isn't it?} 
\end{thm}
Given Theorem \ref{thm1:main}, the proof is nearly the same as the one of the main Theorem~\cite[Thm.1.1]{ES15}, with some differences which we explain in Section~\ref{s:thm1.1}. 

\medskip

The second main theorem is de Jong's conjecture for 
convergent isocrystals coming from geometry, which generalizes \cite[Thm. 2.14]{ES15}. 
Recall that, for a smooth proper morphism $f: Y \to X$ of varieties over a perfect field 
$k$ of characteristic $p>0$, the Gau{\ss}-Manin convergent isocrystal $R^if_* \sO_{Y/K}$ is 
defined by Ogus \cite{Ogu84}. 

\begin{thm} \label{thm2:main}
Let $f: Y\to X$ be a smooth proper morphism between smooth proper varieties over a perfect field $k$ of 
characteristic $p>0$.  If the \'etale fundamental group of $X \otimes_k \overline{k}$ is trivial, 
then the Gau{\ss}-Manin convergent isocrystal $R^if_* \sO_{Y/K}$ 
is isomorphic to $\sO_{X/K}^{\oplus r}$, where $r$ is its rank. 
\end{thm}

\begin{remark} \label{rem:thm2}
For a smooth projective morphism  $f: Y\to X$ between varieties defined over a perfect field $k$ of 
characteristic $p > 0$, Lazda \cite[Cor.~5.4]{Laz15} proved recently that the Gau{\ss}-Manin convergent isocrystal 
$R^if_* \sO_{Y/K}$ canonically lifts to an overconvergent isocrystal on $X$, 
which we denote by $R^if_* \sO_{Y/K}^{\dagger}$ and call the Gau{\ss}-Manin overconvergent isocrystal. 
When $X$ is smooth, geometrically simply connected, $k$ is a finite field and $p \ge 3$, we prove that 
$R^if_* \sO_{Y/K}^{\dagger}$ is trivial as overconvergent 
isocrystal on $X$.  See Section \ref{s:gm}. 
\end{remark}

 We now explain the methods used in order to prove Theorem \ref{thm1:main}.
 
 \medskip
 
There are two ways to prove the vanishing of $c_i^{dR}(E) \in H^{2i}_{dR}(X/k) \, (i>0)$ 
for a locally free sheaf $E$ equipped with an integrable connection 
which does not use Chern-Weil theory. 
% While Chern-Weil theory  uses  $\sC^\infty$-topology on complex manifolds, the vanishing of $c_i^{dR}(X) \in H^{2i}_{dR}(X/k)$ can be proved algebraically over a characteristic $0$ field $k$ by two methods. 

\medskip

One method uses a {\it modified splitting principle} as developed in  \cite{Esn88} and \cite{Esn00}.  On the projective bundle $\P(E)\xrightarrow{\pi} X$, the integrable connection induces a  differential graded algebra $\Omega^\bullet_\tau$, which is a quotient  $\Omega^\bullet_{\P(E)} \to  \Omega^\bullet_\tau$
 of the de Rham complex, and which cohomologically splits $H^i_{dR}(X/k)$ in $H^i(\P(E), \Omega^\bullet_\tau)$. 
Then one shows that the $\Omega^1_\tau$-connection on $\pi^*E$ induced by $\nabla$ stabilizes $\sO_{\P(E)}(1)$. 
Hence, when the rank of $E$ is two, $\pi^*E$ is an extension of $\Omega^1_\tau$-connections of rank $1$ and 
so one can prove the vanishing of the classes in $H^{2i}(\P(E), \Omega^\bullet_\tau)$, thus in $H^{2i}_{dR}(X/k)$. 
In the case of rank $r$, one repeats the above argument $(r-1)$-times to obtain a filtration by 
$\Omega^1_\tau$-connections of rank $1$. 

% In the locally free case, 
We adapt this construction to the crystalline case as follows. 
(In the introduction, we assume the existence of a closed embedding $X \hookrightarrow P$ of $X$ into 
a smooth $p$-adic formal scheme $P$ over $W$ to ease the explanation.)  
One considers the projective bundle $\pi: \P(E_D)\to D$ of the value $E_D$  of a locally free crystal $E$ on the PD-hull 
$X\hookrightarrow D\to \P_W$ of the embedding $X \hookrightarrow P$.  One shows that  the connection on $E_D$ induces a quotient differential graded algebra $\bar  \Omega^\bullet_{\P(E_D)} \to
 \bar  \Omega^\bullet_\tau$ of the PD-de Rham complex $\bar  \Omega^\bullet_{\P(E_D)}$. 
Then one shows that the $\bar \Omega^1_\tau$-connection on $\pi^*(E_D)$ respects $\sO_{\P(E_D)}(1)$. 
Thus we can argue as in the de Rham case and obtain the required vanishing. 
% This kills the classes in $H^{2i}(\P(E_D), \bar \Omega^\bullet_\tau)$, thus in $H^{2i}(X/W)$. 
See Section~\ref{s:spl}.

\medskip

Another way on the de Rham side is to say that local trivializations of $E$ yield a simplicial scheme $X_\bullet$ augmenting to $X$
together with a morphism $e: X_\bullet  \to BGL(r)$, defined by the transition functions, to the simplicial classifying scheme $BGL(r)$, where $r$ is the rank of $E$. This induces the maps $H^{2i}(BGL(r), \Omega^{\bullet}) \to H^{2i}_{dR}(X_{\bullet}/k) 
\cong H^{2i}_{dR}(X/k)$. If $k = \C$ and $E$ carries an integrable connection, we have a similar map  
$H^{2i}(BGL(r), \Omega^{\bullet}) \to H^{2i}_{dR}(X_{{\rm an}, \bullet}/k) 
\overset{\cong}{\leftarrow} H^{2i}_{dR}(X_{\rm an}/k)$ which is identified with the previous one and factors through 
the cohomology $H^{2i}(BGL(r)_{\rm disc}, \Omega^\bullet)=H^{2i}(BGL(r)_{\rm disc}, \sO)$ of 
the discrete classifying simplicial space $BGL(r)_{\rm disc}$. 
Thus $c_i^{dR}(E)$ is in the image of the composite map
 $H^{2i}(BGL(r), \Omega^{\ge i}) \to H^{2i}(BGL(r), \Omega^\bullet) \to H^{2i}(BGL(r), \sO) \to 
H^{2i}(BGL(r)_{\rm disc}, \sO)$, which is zero for $i\ge 1$. 

%In the locally free case, 
We adapt this construction to the crystalline case as follows. 
Given a closed emdedding $X\subset P$ into 
a smooth $p$-adic formal scheme over $W$, one defines $D_\bullet$ to be the simplicial scheme such that $D_n$ is the PD-hull of the diagonal in $P^{\times n+1}$.
Then the crystalline Poincar\'e lemma and the \v{C}ech-Alexander resolution 
equate $H^i_{\rm crys}(X/W)$ both with $H^i(D_\bullet, \bar \Omega^\bullet)$ and with $H^i(D_\bullet, \sO)$ (see Proposition~\ref{prop:OmegaO}).  Thus,  defining a certain simplicial version $\sD_\bullet$ of $D_\bullet$, to which $X_\bullet$ maps, 
$E$ has the free value $E_{\sD_\bullet}$ and this induces the map 
$H^{2i}_{\rm crys}(BGL(r)/W) = H^{2i}(BGL(r), \Omega^{\bullet}) \to H^{2i}(\sD_\bullet, \bar \Omega_{\sD_\bullet}) \overset{\cong}{\to} H^{2n}(\sD_\bullet, \sO) \cong H^{2i}_{\rm crys}(X/W)$. 
Thus $c_i^{\rm crys}(E_X)$ is in the image of the composite map
 $H^{2i}(BGL(r), \Omega^{\ge i}) \to H^{2i}(BGL(r), \Omega^{\bullet}) \to H^{2i}(BGL(r), \sO) \to 
H^{2i}(\sD_{\bullet}, \sO) \cong H^{2i}_{\rm crys}(X/W)$, which is zero for $i\ge 1$. 
%factors through $H^{2i}(\sD/W)=H^{2n}(\sD_\bullet, \sO)$, thus through $H^{2i}(BGL(r), \sO)$, which is trivial.  
See Section~\ref{s:bg}.

% In the non-locally free case, we use the universal $p$-adic enlargement $T$ of $X$ constructed by Ogus  \cite[Prop.~2.3]{Ogu84}, and on it we construct a finite resolution $E_\bullet$  of the value $E_T$ on $T$ of the crystal $E$ by locally free sheaves. Then the crystalline Chern classes of of $E_X$ are computed as the alternate sum of the crystalline Chern classes of $E_n|_X$, which is the same as the alternate sum of the de Rham Chern classes of the  $E_{n, K}$, where  $K$ is the field of fractions of $W$, which is the same as the de Rham Chern class of $E_K$. Using again a $\tau$-splitting principle on the rigid analytic space $T_K$, one shows that these latter classes are zero. 
% See Section~\ref{s:Xproj}.
 
 \medskip
 
 We now explain the methods used in order to prove Theorem \ref{thm2:main}, assuming $k$ is a finite field.
 
  Let us assume first  that $f: Y \to X$ is an abelian scheme and present then an $\ell$-adic argument due to G. Faltings. The arithmetic fundamental group of $X$ acts on $R^if_*\Q_\ell$ via 
 ${\rm Gal}( \bar \F_q/\F_q)$, thus by Tate's theorem \cite{Tat66}, all fibers of $f$ over $\bar \F_p$-points of $X$ are isogeneous, thus the Gau{\ss}-Manin convergent isocrystal $R^if_*\sO_{Y/K}$ is trivial. 
%by the comparison isomorphism \cite{KM74}  between $\ell$-adic and crystalline cohomology,   the latter is isotrivial. 
 
 In general, one has to replace Tate's motivic theorem by a result of Chiarellotto-Le Stum
\cite{CLS98} (generalizing 
Katz-Messing theorem \cite{KM74}) and Abe's 
 \v{C}ebotarev density theorem \cite[Prop.A.3.1]{Abe13}, which yields the triviality of the semi-simplification of 
 $R^if_*\sO_{Y/K}^{\dagger}$ (see Remark \ref{rem:thm2} for notation) 
in the category ${\rm Conv}^{\dagger}(X/K)$ of overconvergent isocrystals on $X$. 
%,  as we know that the Gau{\ss}-Manin convergent isocrystal is coming from a $F$-overconvergent isocrystal \cite[Cor.~5.4]{Laz15}.  
Finally one has to go from the semi-simplifcation to the original Gau{\ss}-Manin isocrystal by 
showing that there are no extensions of the trivial overconvergent isocrystal by itself on $X$ 
when $X$ is proper or $p \geq 3$ (Theorem~\ref{thm:van}). 
 
Over a non-finite field (when $X$ is proper), one reduces the proof to the case of a finite ground field by a specialization argument. 
See Section~\ref{s:gm}.

\medskip

Finally, in Section~\ref{s:lefschetz} we prove a very weak form of a Lefschetz theorem for isocrystals. 
 
 \medskip
 {\it Acknowledgements:} It is a pleasure to thank Gerd Faltings for his interest in our work and for a discussion in May 2015 in Beijing, which helped us proving Theorem~\ref{thm2:main}.   In addition, we thank  Bhargav Bhatt, Johan de Jong, Peter Scholze for discussions on isocrystals, notably on the question of whether they always carry a locally free lattice.  Unfortunately, we still do not have a general   answer to this basic question. Nikita Markaryan kindly sent us his non-posted notes \cite{Mar11} which aim at constructing a crystalline Chern-Weil theory, unfortunately the notes stop before this topic is addressed. 
 
\section{Crystalline modified splitting principle} \label{s:spl}

The aim of this section is to prove Theorem~\ref{thm1:main} using a crystalline modified splitting principle.

\medskip

Let $X$ be a smooth variety over a perfect field $k$ of characteristic $p>0$ and 
let $X_{(\bullet)} \to X$ be a simplicial scheme augmented to $X$ defined 
as the \v{C}ech hypercovering associated to an open covering $X= \bigcup_{i \in I}X_i$ 
which admits a closed embedding $X_{(\bullet)} \to D_{(\bullet)}$ into 
a simplicial $p$-adic formal scheme $D_{(\bullet)}$ over $W$
such that, for any $n \in \N$ and Zariski locally on $X_{(n)}$, 
$X_{(n)} \to D_{(n)}$ is the PD-envelope of a closed immersion $X_{(n)} \to Y$ of $X_{(n)}$ into 
a smooth $p$-adic formal scheme $Y$ over $W$. 
(Note that, for any smooth variety $X$ over $k$, there exists such a system.) 

\subsection{Generalities on the first crystalline Chern class, and variants of it} \label{ss:first}
 One denotes by $(X/W)_{\rm crys}, (X_{(\bullet)}/W)_{\rm crys}$ 
the crystalline topos of $X/W$, $X_{(\bullet)}/W$ respectively 
 (see \cite[p.5.3]{BO78}), by $(X_{(\bullet)}/W)_{\rm crys}|_{D_{(\bullet)}}$ its restriction to $D_{(\bullet)}$ (\cite[p.5.22]{BO78}), by $(-)_{\rm Zar}$ the Zariski topos.  In particular the canonical morphism of topo\"i $X_{{(\bullet)}\rm Zar}\to D_{{(\bullet)}\rm Zar}$ 
is an equivalence. 
On $D_{{(\bullet)}\rm Zar}$ one defines the  PD-de Rham complex  $\bar \Omega^\bullet_{D_{(\bullet)}}$, 
which is a quotient differential graded algebra of the de Rham complex $\Omega^\bullet_{D_{(\bullet)}}$ of $D_{(\bullet)}$.  
The submodule $\sK\subset \Omega^1_{D_{(\bullet)}}$,  topologically spanned by $da^{[m]}-a^{[m-1]}da \, 
(a \in I={\rm Ker}(\sO_{D_{(n)}}\to \iota_* \sO_{X_{(n)}}), n,m \in \N)$ spans all relations, that is 
\ga{1}{ \bar \Omega^\bullet_{D_{(\bullet)}} =\Omega^\bullet_{D_{(\bullet)}}/\sK\wedge \Omega^{\bullet -1}_{D_{(\bullet)}}}
and in addition, each $\bar \Omega^i_{D_{(n)}}$ is locally free over $\sO_{D_{(n)}}$ with the relation
\ga{2}{\bar \Omega^i_{D_{(\bullet)}}=\wedge^i_{\sO_{D_{(\bullet)}}} \bar \Omega^1_{D_{(\bullet)}}}
(\cite[Prop. ~3.1.6]{Ill79}). 

One has the following commutative diagram of topo\"i \cite[p.6.12]{BO78}
\ga{1}{   \xymatrix{ \ar[d]_j (X_{(\bullet)}/W)_{\rm crys}|_{D_{(\bullet)}} \ar[r]^{\varphi} & D_{{(\bullet)}\rm Zar}=X_{{(\bullet)}\rm Zar}. \\
(X_{(\bullet)}/W)_{\rm crys}\ar[ur]_u
}
}
The complex
\ga{3}{L(\bar \Omega^\bullet_{D_{(\bullet)}}):= j_*\varphi^* (\bar \Omega^\bullet_{D_{(\bullet)}})}
in $(X_{(\bullet)}/W)_{\rm crys}$ is defined in \cite[p.6.13]{BO78}, and it is proved in \cite[Thm.6.12]{BO78} that the natural map 
\ga{4}{ \sO_{X_{(\bullet)}/W}\to L(\bar \Omega^\bullet_{D_{(\bullet)}})}
is a quasi-isomorphism in $(X_{(\bullet)}/W)_{\rm crys}$. As a consequence, restricting \eqref{4} to the crystalline units, one obtains a quasi-isomorphism 
in $X_{{(\bullet)}\rm Zar}$
\ga{5}{Ru_* \sO^\times_{X_{(\bullet)}/W} \xrightarrow{\cong} (\sO^\times_{D_{(\bullet)}} \xrightarrow{d \log} \bar \Omega^1_{D_{(\bullet)}}\xrightarrow{d}  \bar \Omega^2_{D_{(\bullet)}}\to \ldots).}
The exact sequence 
\ga{6}{ 0\to \sI_{X_{(\bullet)}/W}\to \sO_{X_{(\bullet)}/W}\to \iota_* \sO_{X_{(\bullet)}}\to 0} in $(X_{(\bullet)}/W)_{\rm crys}$, defining $\sI_{X_{(\bullet)}/W}$, 
 yields an exact sequence
\ga{7}{ 1\to (1+ \sI_{X_{(\bullet)}/W}) \to \sO^\times _{X_{(\bullet)}/W}\to \iota_* \sO^\times _{X_{(\bullet)}}\to 1
} in $(X_{(\bullet)}/W)_{\rm crys} $.
By \cite[2.1]{BI70} and the functoriality of the construction there, 
the connecting homomorphism $$H^1(X, \sO^\times_X)\to H^1(X_{(\bullet)}, \sO^\times_{X_{(\bullet)}}) \to 
H^2((X_{(\bullet)}/W)_{\rm crys}, 1+\sI_{X_{(\bullet)}/W}),$$ followed by the logarithm $$H^2((X_{(\bullet)}/W)_{\rm crys}, 1+\sI_{X_{(\bullet)}/W}) \to H^2((X_{(\bullet)}/W)_{\rm crys}, \sI_{X_{(\bullet)}/W})$$ 
and the natural map $$H^2((X_{(\bullet)}/W)_{\rm crys},  \sI_{X_{(\bullet)}/W})\to H^2_{\rm crys}(X_{(\bullet)}/W) \cong 
H^2_{\rm crys}(X/W) $$  precisely computes $c_1^{\rm crys}$. 
Thus applying $Ru_*$ to \eqref{7} and using \eqref{5},  we conclude that the connecting homomorphism 
$H^1(X, \sO^\times_X)\to H^1(X_{(\bullet)},\sO^\times_{X_{(\bullet)}})\to 
 H^2(D_{(\bullet)}, (1+I_{(\bullet)})\xrightarrow{d \log} \bar  \Omega^{\ge 1}_{D_{(\bullet)}} ))$ of the exact sequence  in 
$D_{{(\bullet)}\rm Zar}$
\ga{8}{1\to ((1+I_{(\bullet)})\xrightarrow {d \log}  \bar \Omega^{\ge 1}_{D_{(\bullet)}} ) \to ( \sO^\times_{D_{(\bullet)}} \xrightarrow {d \log} \bar  \Omega^{\ge 1}_{D_{(\bullet)}}) \to \sO^\times_X\to 1, }
followed by the logarithm $H^2(D_{(\bullet)}, (1+ I_{(\bullet)}) \xrightarrow {d \log}  \bar \Omega^{\ge 1}_{D_{(\bullet)}}) \to 
H^2(D_{(\bullet)}, I \xrightarrow{d}  \bar \Omega^{\ge 1}_{D_{(\bullet)}})$  and the natural map $H^2(D_{(\bullet)},  I \xrightarrow{d} \bar  \Omega^{\ge 1}_{D_{(\bullet)}} )\to H^2(D_{(\bullet)}, \bar \Omega^\bullet_{D_{(\bullet)}})=H^2_{\rm crys}(X_{(\bullet)}/W) = 
H^2_{\rm crys}(X/W)$ precisely computes $c_1^{\rm crys}$.  In particular, if $E_X$ is a locally free sheaf of rank $1$ on $X$, then by \eqref{5}, $c_1^{\rm crys}(E_X)=0$ if $E_X$ is the value on $X$ of a locally free crystal of rank $1$ on $X$. 

\medskip

More generally, let $ \tau: \bar \Omega^\bullet_{D_{(\bullet)}}\surj A^\bullet_{(\bullet)}$ is be a surjection of sheaves of differential graded algebras on $D_{{(\bullet)}\rm Zar}$, such that, for any $n \in \N$, $A^1_{(n)}$ is locally free over $\sO_{D_{(n)}}$, 
$A^0_{(n)}=\sO_{D_{(n)}}$,  $A^m_{(n)}=\wedge^m_{\sO_{D_{(n)}}} A^1_{(n)}$. A  $\tau$-{\it connection} (see \cite[(2.1)]{Esn88}) on a locally free sheaf $E_{D_{(\bullet)}}$  on $D_{(\bullet)}$ is an additive map $\nabla_\tau: E_{D_{(\bullet)}}\to A^1_{(\bullet)}
\otimes_{\sO_{D_{(\bullet)}}}  E_{D_{(\bullet)}}$ which fulfills the $\tau$-{\it Leibniz} rule $\nabla_{\tau}(\lambda  e)=\tau d(\lambda) \otimes e+ \lambda \otimes \nabla_\tau(e)$. Then $\nabla_\tau \circ \nabla_\tau: E_{D_{(\bullet)}}\to A^2_{(\bullet)} \otimes E_{D_{(\bullet)}}$, where $\nabla_\tau( \omega\otimes e)= (-1)^i \tau d(\omega) + \omega\otimes \nabla_\tau (e )$ for $\omega\in A^i_{(n)}$,  is $\sO_{D_{(\bullet)}}$-linear. The
$\tau$-connection is {\it integrable} if $\nabla_\tau \circ \nabla_\tau=0$.
One pushes down \eqref{8} along $\tau$ and obtains the exact sequence  in $D_{{(\bullet)}\rm Zar}$
\ga{9}{ 1\to ((1+I_{(\bullet)})\xrightarrow {d \log}  A^{\ge 1}_{(\bullet)} ) \to ( \sO^\times_{D_{(\bullet)}} \xrightarrow {d \log} 
A^{\ge 1}_{(\bullet)}) \to \sO^\times_{X_{(\bullet)}}\to 1.}
Then, for a locally free sheaf $E_{X}$ of rank $1$ on $X$,  $0=\tau(c_1^{\rm crys}(E_X)) \in H^2(D_{(\bullet)}, A^\bullet_{(\bullet)})$ if $E_X|_{X_{(\bullet)}}$ is the restriction to $X_{(\bullet)}$ of a line bundle  $E_{D_{(\bullet)}}$ on $D_{(\bullet)}$ which 
is endowed with an integrable $\tau$-connection. 

\subsection{Modified splitting principle} \label{ss:mod_spl_pr}
Let  $ \tau: \bar \Omega^\bullet_{D_{(\bullet)}}\surj A^\bullet_{(\bullet)}$ be a surjection of sheaves of differential graded algebras on $D_{{(\bullet)}\rm Zar}$ 
such that, for any $n \in \N$, $A^1_{(n)}$ is locally free over $\sO_{D_{(n)}}$, $A^0_{(n)}=\sO_{D_{(n)}}$ and 
$A^m_{(n)}=\wedge^m_{\sO_{D_{(n)}}} A^1_{(n)}$. Also, let 
$E_X$ be a locally free sheaf on $X$ such that its restriction $E_X|_{X_{(\bullet)}}$ to $X_{(\bullet)}$ extends to 
a locally free sheaf $E_{D_{(\bullet)}}$ on $D_{(\bullet)}$ endowed with an integrable $\tau$-connection. 
One defines $X'=\P(E_X), X'_{(\bullet)} = \P(E_X|_{X_{(\bullet)}}), \ D'_{(\bullet)}=\P(E_{D_{(\bullet)}})$, together with 
the augumentation $X'_{(\bullet)} \to X'$ and 
the closed embedding $\iota':  X'_{(\bullet)}\to D'_{(\bullet)}$. One has a cartesian squares of (simplicial formal) schemes (over $W$)
\ga{10}{\xymatrix{ X' \ar[d]_\pi & \ar[l] \ar[d]_\pi X'_{(\bullet)} \ar[r]^{\iota'} & D'_{(\bullet)}  \ar[d]_{\pi}\\
X & \ar[l] X_{(\bullet)} \ar[r]_{\iota} & D_{(\bullet)}. 
}}
The PD-structure on $I_{(\bullet)}$ extends uniquely to a PD-structure on $I'_{(\bullet)}={\rm Ker}(\sO_{D'_{(\bullet)}}\to \sO_{X'_{(\bullet)}})$ as $\pi$ is flat (\cite[Prop. ~3.21]{BO78}), and $\iota'$ is again the PD-envelope of a closed immersion from 
$X'_{(n)}$ to a smooth $p$-adic formal scheme over $W$ 
Zariski locally on $X_{(n)}$ for each $n$. Thus one can define $\bar \Omega^\bullet_{D'_{(\bullet)}}$ and one has an exact sequence
\ga{11}{0\to \pi^* \bar \Omega^1_{D_{(\bullet)}}\to \bar \Omega^1_{D'_{(\bullet)}}\to \Omega^1_{D'_{(\bullet)}/D_{(\bullet)}}\to 0}
on $X'_{{(\bullet)}\rm Zar}$. 
Setting $\sK_\tau={\rm Ker}(\bar \Omega^1_{D_{(\bullet)}}\to A^1_{(\bullet)})$, one defines $\bar \Omega^1_{D'_{(\bullet)},\tau}=\bar  \Omega^1_{D'_{(\bullet)}}/\pi^*\sK_\tau$. By definition, \eqref{11} pushs down to an exact sequence
\ga{12}{0\to \pi^* A^1_{(\bullet)}\xrightarrow{i} \bar \Omega^1_{D'_{(\bullet)},\tau} \xrightarrow{p}\Omega^1_{D'_{(\bullet)}/D_{(\bullet)}}\to 0}
on $X'_{{(\bullet)}\rm Zar}$.  One defines $\bar \Omega^n_{D'_{(\bullet)},\tau}=\wedge^n_{\sO_{D'_{(\bullet)}}} \bar \Omega^1_{D'_{(\bullet)},\tau}$. 
By \cite[Claim~p.~332]{Esn88}, the quotient homomorphism $r: \bar \Omega^1_{D'_{(\bullet)}}\to \bar \Omega^1_{D'_{(\bullet)}, \tau}$ extends to a quotient
\ga{13}{r: \bar \Omega^\bullet_{D'_{(\bullet)}}\to \bar \Omega^\bullet_{D'_{(\bullet)},\tau}}
of differential graded algebras on $X'_{{(\bullet)}\rm Zar}$, where the differential on $ \bar \Omega^\bullet_{D'_{(\bullet)},\tau}$  is denoted by $rd$. 

\medskip

The connection $\nabla_\tau$ on $E_{D_{(\bullet)}}$ induces a pull-back connection $\pi^*\nabla_\tau: \pi^*E_{D_{(\bullet)}}\to \bar \Omega^1_{D'_{(\bullet)},\tau} \otimes_{\sO_{D'_{(\bullet)}} }\pi^*E_{D_{(\bullet)}}$. 
Its restriction to $\Omega^1_{D'_{(\bullet)}/D_{(\bullet)}}(1)$ via the exact sequence
\ga{14}{ 0\to \Omega^1_{D'_{(\bullet)}/D_{(\bullet)}}(1) \to \pi^*E_{D_{(\bullet)}}\to \sO_{D'_{(\bullet)}}(1)\to 0,}
followed by the projection $\bar \Omega^1_{D'_{(\bullet)},\tau} \otimes_{\sO_{D'_{(\bullet)}} }\pi^*E_{D_{(\bullet)}} \to 
\bar \Omega^1_{D'_{(\bullet)},\tau} \otimes_{\sO_{D'_{(\bullet)}}} \otimes \sO_{D'_{(\bullet)}}(1)$
defines a section $$\sigma: \Omega^1_{D'_{(\bullet)}/D_{(\bullet)}} \to \bar \Omega^1_{D'_{(\bullet)},\tau}$$ of $-p$ (\cite[(2.4)]{Esn88}).  Thus $\tau'=1+p\circ \sigma : \bar \Omega^1_{D'_{(\bullet)},\tau}\to \pi^* A^1_{(\bullet)}$ is a section of $i$. By \cite[(2.5)]{Esn88}, $\tau'$ induces a surjective homomorphism
\ga{15}{\tau': \bar \Omega^\bullet_{D'_{(\bullet)}, \tau} \surj \pi^*A^\bullet_{(\bullet)}}
of differential graded algebras, where the differential on the right is induced by $\tau' \circ rd$. 

\medskip

The main point is then that the $ \pi^*A^1_{(\bullet)}$-valued connection $\nabla': \pi^*E_{D_{(\bullet)}} \to \pi^*A^1_{(\bullet)}\otimes_{\sO_{D'_{(\bullet)}}} \pi^*E_{D_{(\bullet)}}$, which is $\pi^*\nabla_\tau$ followed by $\tau'$ is integrable and respects the flag \eqref{14}, thus induces integrable connections 
\ga{15}{ \nabla': \sO_{D'_{(\bullet)}}(1)\to \pi^*A^1_{(\bullet)}\otimes \sO_{D'_{(\bullet)}}(1), \ \nabla': \Omega^1_{D'_{(\bullet)}/D_{(\bullet)}}(1) \to \pi^*A^1_{(\bullet)}\otimes \Omega^1_{D'_{(\bullet)}/D_{(\bullet)}}(1).}
Consequently we can iterate the construction, replacing $X$ by $X'$, 
$X_{(\bullet)}$ by $X'_{(\bullet)}$, 
$D_{(\bullet)}$ by $D'_{(\bullet)}$, $E_X$ by the descent of 
$\Omega^1_{D'_{(\bullet)}/D_{(\bullet)}}(1)|_{X'_{(\bullet)}}$ to $X'$, $E_{D_{(\bullet)}}$ by  $\Omega^1_{D'_{(\bullet)}/D_{(\bullet)}}(1)$, $\tau$ by $\tau' \circ r$. 

\subsection{Proof of Theorem~\ref{thm1:main}}

Further iterating, after $(r-1)$-steps, one obtains a diagram as \eqref{10} where now $X'$ is the complete flag bundle over $X$, with \eqref{15} becoming a surjective homomorphism 
\ga{16}{ \tau:  \bar  \Omega^\bullet_{D'_{(\bullet)}}\to \pi^*\bar \Omega^\bullet_{D_{(\bullet)}}}
of differential graded algebras, 
 and with a filtration on $\pi^*E_{X}$ with graded being a sum of locally free sheaves $L_j$ of rank $1$ such that 
the restriction $L_j|_{X'_{(\bullet)}}$ of $L_j$ to $X'_{(\bullet)}$ extends to a locally free sheaf on 
$D'_{(\bullet)}$ endowed with an integrable $\pi^*\bar \Omega^1_{D_{(\bullet)}}$-connection. 
In addition, from \cite[Lemma~in~(1.3), (2.7)]{Esn88}, the composite
\ga{17}{ \bar \Omega^\bullet_{D_{(\bullet)}} \to R\pi_* \bar \Omega^\bullet_{D'_{(\bullet)}}\xrightarrow{R\pi_*\tau} R\pi_* 
\pi^*  \bar \Omega^\bullet_{D_{(\bullet)}} \xleftarrow{\cong} \bar \Omega^\bullet_{D_{(\bullet)}}}
is the identity on $X_{{(\bullet)}\rm Zar}$. 
By the standard  Whitney product formula for crystalline Chern classes \cite[III Thm. ~1.1.1]{Gro85} 
%\cite[Section~5]{Pet03}, 
for $i\ge 1$, $c_i^{\rm crys}(\pi^*E_X)$ is a sum of products of $c_1^{\rm crys}(L_j)$, which by Subsection~\ref{ss:first} maps to $0$ in $H^2(D'_{(\bullet)}, \pi^*  \bar \Omega^\bullet_{D_{(\bullet)}})$. Thus 
%, as crystalline Chern classes are functorial, 
$c_i^{\rm crys}(E_X) \in H^{2i}(D_{(\bullet)}, \bar \Omega^\bullet_{D_{(\bullet)}})$ maps to $0$ in $H^{2i}(D_{(\bullet)},  \bar \Omega^\bullet_{D_{(\bullet)}})$ via the composite map in \eqref{17}, which is the identity. This shows  Theorem~\ref{thm1:main}.

\begin{remark}\label{rem:torsion1} 
Let $W_n := W/p^nW$. Then, by replacing $W$ by $W_n$ and 
$D_{(\bullet)}$ by its mod $p^n$ reduction, we see that the proof above 
gives the following variant of Theorem \ref{thm1:main}: 
If $X$ is as in Theorem \ref{thm1:main} and
if $E$ is a locally free crystal on $X/W_n$, 
then $c_i^{\rm crys}(E_X)$ is zero in the torsion crystalline cohomology group 
$H^{2i}_{\rm crys}(X/W_n)$ for $i\ge 1$. Because 
$H^{2i}_{\rm crys}(X/W) \otimes_W W_n \to H^{2i}_{\rm crys}(X/W_n)$ is injective, 
it implies that the Chern classes $c_i^{\rm crys}(E_X)$ in $H^{2i}_{\rm crys}(X/W) \,(i \ge 1)$ 
are divisible by $p^n$ in this case. 
\end{remark}

\subsection{Remark on Chern-Simons theory}
In  \cite{Esn92} and \cite{Esn00}, a version of the modified splitting principle which is slightly more elaborate than the one used in  \ref{ss:mod_spl_pr} was performed in order to construct classes  $c_i(E,\nabla) \in H^i(X, \sK^M_i\xrightarrow{d \log} \Omega^i_X \xrightarrow{d}  \ldots)$ of a  bundle with an integrable connection $(E,\nabla)$, depending on $\nabla$, where $\sK^M_i$ is the Zariski sheaf of Milnor $K$-theory. Those classes  lift both the Chow classes in $CH^i(X)=H^i(X, \sK_i^M)$ via the obvious forgetful map, and the Chern-Simons  classes in $H^{2i-1}(X_{\rm an}, \C/\Z(i))$, if $k=\C$.  We hope to be able to define a crystalline version of Chern-Simons theory, yielding classes lifting both Chow classes and classes in syntomic cohomology. 

\section{The crystalline version of the discrete classifying space $BGL(r)$}\label{s:bg}
The aim of this section is to prove Theorem~\ref{thm1:main} using a crystalline version of the discrete classifying space $BGL(r)$. 

Let $X$ be a smooth variety defined over a perfect field $k$ of characteristic $p>0$.
\subsection{\v{C}ech-Alexander resolutions of $\sO_{X/W}$ }  \label{ss:cech}

Fix a closed embedding $X\hookrightarrow Y$ into a $p$-adic smooth formal scheme $Y$ over $W$, and one defines $D(n)$ as the PD-envelope of $X$ in the diagonal embedding $Y^{n+1}$, where $^{n+1}$ means the product over $W$. 
%One denotes by $(X/W)_{\rm crys}$ the crystalline topos of $X$ relative to $W$
% (see  \cite[p.5.3]{BO78}), by $(X/W)_{\rm crys}|_{D(n)}$ its restriction to $D(n)$ (\cite[p.5.22]{BO78}). 
Then one has the canonical morphism of topo\"i 
  \ga{17.1}{ j_n: (X/W)_{\rm crys}|_{D(n)} \xrightarrow{} (X/W)_{\rm crys}}
and the \v{C}ech-Alexander resolution \cite[5.29]{BO78} of the abelian sheaf $\sO_{X/W}$
  \ga{17.2}{\sO_{X/W} \xrightarrow{\cong} (j_{0*}j_0^* \sO_{X/W} \to j_{1*}j_1^* \sO_{X/W} \to \ldots ).} 

We also use the following variant of \eqref{17.2}. Fix two closed embeddings $X\hookrightarrow Y$ and $X\hookrightarrow Z$, where both $Y/W$ and $Z/W$ are formal schemes and $Y/W$ is smooth, one defines $D(n)$ as the PD-envelope of $X$ in the diagonal embedding $Y^{n+1}\times_W Z$. Then one has the canonical morphisms of topo\"i
  \ga{18}{ j_n: (X/W)_{\rm crys}|_{D(n)} \xrightarrow{\ell_n} (X/W)_{\rm crys}|_{D(-1)} \xrightarrow{j_{-1}}
 (X/W)_{\rm crys}.}
  On $(X/W)_{\rm crys}|_{D(-1)}$, one has the \v{C}ech-Alexander resolution \cite[5.29]{BO78} of the abelian sheaf $j_{-1}^*\sO_{X/W}$
  \ga{19}{j_{-1}^*\sO_{X/W} \xrightarrow{\cong} (\ell_{0*}j_0^* \sO_{X/W} \to \ell_{1*}j_1^* \sO_{X/W} \to \ldots ).}
Hence, applying the exact functor $j_{-1*}$ (\cite[Cor.5.27.1]{BO78}), one obtains the resolution of the abelian sheaf $j_{-1*}j_{-1}^*\sO_{X/W}$
\ga{20}{j_{-1*} j_{-1}^*\sO_{X/W} \xrightarrow{\cong} (j_{0*}j_0^* \sO_{X/W} \to j_{1*}j_1^* \sO_{X/W} \to \ldots ).}

\subsection{Various simplicial constructions to compute crystalline cohomology} \label{ss:crys}
Let $X=\cup_{i\in I} X_i$ be a finite  covering of $X$ by affine open subvarieties. We assign to it the standard Mayer-Vietoris simplicial scheme, the definition of which we recall now. 

\medskip

We choose a total order on $I$, define the set of tuples $I^n_{\le }:= \{(i_0,\ldots, i_n)\  ; \  i_0\le i_1\le \ldots \le i_n)\}$, and $I_{\le }$ to be the disjoint union of the $I^n_{\le }$. For $J =(i_0, \ldots, i_n) \in I_{\le }$, one sets $X_J=\cap_{i_j\in J} X_{i_j}$. One upgrades $I_{\le }$ to a category. The Hom-set
${\rm Hom}_{I_{\le }}( J, J')$, for  $ J=(i_0,\ldots, i_n)$ and  $ J'=(i'_0,\ldots, i'_{n'})$,  consists of those non-decreasing maps $\varphi: [n]\to [n']$, where $[n]=\{0,1,\ldots, n\}$,  with the property that $i_a=i'_{\varphi(a)}$ for all $a\in \{0,\ldots, n\}$. Thus to $\varphi \in {\rm Hom}_{I_{\le }}(J, J')$,  one assigns
 the open embedding $ X_{J'} \hookrightarrow X_J$,  which one denotes by $\varphi^*$. 

If $\varphi \in {\rm Hom}_{I_{\le }} (J, J')$ with $J'=(i'_0,\ldots, i'_{n'})$,
$J$ is necessarily equal to $(i'_{\varphi(0)}, \ldots, i'_{\varphi(n)})$. Thus, given any non-decreasing map $\varphi: [n]\to [n']$ and $J' \in I_{\le}^{n'}$,  there is one and only one $J \in I_{\le}^n$ such that $\varphi \in {\rm Hom}_{I_{\le }} (J, J')$, in particular the open embedding $\varphi^*: X_{J'} \to X_J$ is determined as well.  Denoting by $X_{(n)} =\sqcup_{J\in I^n_{\le }} X_J$ the disjoint union of the $X_J$ over all the  $J\in I^n_{\le }$, one defines the map 
\ga{21}{ \varphi^*  : X_{(n')}\to X_{(n)} }
for a non-decreasing map $\varphi: [n]\to [n']$ as the disjoint union of the maps $X_{J'} \xrightarrow{\varphi^*} X_J \hookrightarrow X^{(n)}$ 
for $J' \in I^{n'}_{\le}$. Using the definition \cite[Section~5]{Del74},  \eqref{21} defines the simplicial scheme $X_{(\bullet )}$  which augments to $X$
\ga{22}{X_{(\bullet)} \xrightarrow{\epsilon} X.}

\begin{remark}
The simplicial scheme $X_{(\bullet)}$ here differs from 
the simplicial scheme $X_{(\bullet)}$ which appeared in Section \ref{s:spl}.

\end{remark}

The aim of this section is to prove the 
\begin{prop} \label{prop:crys_simpl}
The augmentation map induces a quasi-isomorphism 
\ga{23}{\epsilon^*: R\Gamma((X/W)_{\rm crys}, \sO_{X/W}) \to  R\Gamma((X_{(\bullet)}/W)_{\rm crys}, \sO_{X_{(\bullet )}/W})}
in the derived category  $D({\rm Ab})$ of abelian groups. 
\end{prop}
\begin{proof}
The proof goes by induction on the cardinality  $|I|$ of $I$. If $|I|=1$,   one sets $G=R\Gamma((X/W)_{\rm crys}, \sO_{X/W}) \in D({\rm Ab})$. Then the right hand side of \eqref{23} reads
\ga{24}{ G\xrightarrow{ \alpha_0}  G   \to \ldots \xrightarrow{\alpha_n}  G \to \ldots }
where  $\alpha_n= {\rm id}$ for $n$ even, and $\alpha_n=0$ for $n$ odd. 

For $|I|>1$, we subdivide the simplicial construction $X_{(\bullet)}$ as follows.  Let $0$ be the minimal element of $I$. One sets $=I'\sqcup I''$, with $I'=\{0\}$ and $I''=I\setminus I'$.  Then one has
\ga{25}{ I_{\le }=I'_{\le } \sqcup (\sqcup_{ (n,m) \in \N} {I}^{'n}_{\le }\times {I}^{''m}_{\le })\sqcup I''_{\le }.}

Setting $X_{J',J''}=X_{J'}\cap X_{J''}$ for  $J'\in I'_{\le }$ and $J''\in I''_{\le }$, one sets $X_{(n,m)} =\sqcup_{ J'\in I_{\le }^{'n}, J'' \in I_{\le }^{''m}  } X_{J',J''}$. Then $X_{(\bullet, \bullet)}$ forms a bisimplicial scheme (\cite[p.~17]{Fri82}) and one has the commutative diagram
\ga{26}{\xymatrix{ \ar[d]_{\epsilon '} X'_{(\bullet)} & \ar[l] X_{(\bullet, \bullet)} \ar[r] & X''_{(\bullet )} \ar[d]^{\epsilon ''} \\
X'  &  &  X''}
}
with $X'=X_0 $ and $X''=\cup_{i\in I\setminus \{0\}} X_i$.
By induction, Proposition~\ref{prop:crys_simpl} applies to $\epsilon'$ and $\epsilon''$. On the other hand, $X_{(\bullet ,m)}$ is the constant simplicial scheme on $X' \cap X''_{(m)}$.  So by the case $|I|=1$, one has
\ga{27}{ R\Gamma((X'\cap X''_{(m)}/W)_{\rm crys}, \sO_{ X' \cap X''_{(m)} /W }) \xrightarrow{\cong} R\Gamma((X_{(\bullet,m)}/W)_{\rm crys}, \sO_{ X_{(\bullet,m)}/W  }).}
From this and the induction hypothesis one deduces the isomorphism
\begin{align}
\epsilon_1^*: 
R\Gamma((X' \cap X''/W)_{\rm crys}, \sO_{ X' \cap X''/W }) & \xrightarrow{\cong}
R\Gamma((X' \cap X''_{(\bullet)}/W)_{\rm crys}, \sO_{ X' \cap X'' _{(\bullet)} /W })  \label{28} \\ & 
\xrightarrow{\cong} R\Gamma((X_{(\bullet,\bullet)}/W)_{\rm crys}, \sO_{ X_{(\bullet,\bullet)}/W  }). \nonumber 
\end{align}
One now extends \eqref{23} to a diagram 
\ga{29}{\tiny{ \xymatrix{\ar[d]  R\Gamma((X/W)_{\rm crys}, \sO_{X/W}) \ar[r]^{\epsilon^*}  &  R\Gamma((X_{(\bullet)}/W)_{\rm crys}, \sO_{X''_{(\bullet )}/W}) \ar[d] \\
\ar[d] R\Gamma((X'/W)_{\rm crys}, \sO_{X'/W}) \oplus  R\Gamma((X''/W)_{\rm crys}, \sO_{X''/W}) \ar[r]^{\epsilon^{'*} \oplus \epsilon^{''*}} &
 R\Gamma((X'_{(\bullet)}/W)_{\rm crys}, \sO_{X'_{(\bullet )}/W}) \oplus R\Gamma((X''_{(\bullet)}/W)_{\rm crys}, \sO_{X''_{(\bullet )}/W})  \ar[d] \\
 R\Gamma((X' \cap X''/W)_{\rm crys}, \sO_{ X' \cap X'' /W })  \ar[r]^{\epsilon_1^*} & R\Gamma((X_{(\bullet,\bullet)}/W)_{\rm crys}, \sO_{ X_{(\bullet,\bullet)}/W  })
 } }}
where the left vertical triangle is induced by the quasi-isomorphism 
in \cite[V (3.5.4)]{Ber74} and the right one exists by construction. 
% are   the van Kampen {{\red right terminology?}} exact triangles. 
As $\epsilon^{'*}, \epsilon^{''*}, \epsilon_1^*$ are isomorphisms, so is $\epsilon^*$.
\end{proof} 

\subsection{Lifting the simplicial construction to PD-envelopes}
Keeping the same  notations, we choose for each affine $X_i$ a closed embedding $X_i\hookrightarrow Y_i$ into a 
smooth $p$-adic formal scheme $Y_i$ over $W$.  One defines the PD-envelope $\alpha_J: X_J\hookrightarrow D_J$  of $X_J \hookrightarrow Y_{i_0}\times_W \ldots \times_W Y_{i_n}$ for $J=(i_0,\ldots, i_n)\in I^n_{\le }$. As in \eqref{22} and \eqref{26} one has the simplicial formal scheme $D_{(\bullet)}$ and 
the diagram of 
simplicial formal schemes
\ga{30}{{\xymatrix{  D'_{(\bullet)} & \ar[l] D_{(\bullet, \bullet)} \ar[r] & D''_{(\bullet )}}
}}
this time without augmentation. 
On the other hand, for each $J\in I_{\le }$, one has as in \eqref{1} the diagram of topo\"i
\ga{31}{   \xymatrix{ \ar[d]_{j_J} (X_J/W)_{\rm crys}|_{D_J} \ar[r]^{\varphi_J} & D_{J,\rm Zar}=X_{J,\rm Zar}\\
(X_J/W)_{\rm crys}\ar[ur]_{u_J}
}
}
to which one applies \eqref{3} and the quasi-isomorphism \eqref{4}, which  in addition is functorial.  Thus, combined with  \eqref{23},  it yields quasi-isomorphisms
\ml{32}{ R\Gamma((X/W)_{\rm crys}, \sO_{X/W})\xrightarrow{\cong} R\Gamma((X_{(\bullet)}/W)_{ \rm crys}, \sO_{X_{(\bullet)}/W}) )\xrightarrow{\cong} \\
R\Gamma((X_{(\bullet)}/W)_{ \rm crys}, L( \bar  \Omega^\bullet_{D_{(\bullet)}})) \xleftarrow{\cong} 
R\Gamma(D_{(\bullet)}, \bar  \Omega^\bullet_{D_{(\bullet)}}). 
} 
(The last isomorphism follows from \cite[Cor.~5.27.2]{BO78}.) 
\begin{prop}  \label{prop:OmegaO}
The forgetful morphism
\ga{33}{ R\Gamma (D_{(\bullet)}, \bar \Omega^\bullet_{D_{(\bullet)}}) \to 
R\Gamma (D_{(\bullet)}, \sO_{D_{(\bullet)}})}
is a quasi-isomorphism. 
\end{prop}
\begin{proof}
%Precomposing with  $ \sO_{D_{(\bullet)}}$ and postcomposing with $ \sO_{D_{(\bullet)}} \to L(\sO_{D_{(\bullet)}})$, 
It is enough to show that the map
\ga{34}{ R\Gamma((X/W)_{\rm crys}, \sO_{X/W}) \to R\Gamma((X_{(\bullet)}/W)_{{\rm crys}}, L(\sO_{D_{(\bullet)}}))}
is a quasi-isomorphism. 
As for Proposition~\ref{prop:crys_simpl}, we argue by induction in $|I|$. For  $|I|=1$, the right hand side is $R\Gamma((X/W)_{\rm crys}, (-))$ of the right hand side of \eqref{17.2}, thus computes $R\Gamma((X/W)_{\rm crys}, \sO_{X/W})$. For general $I$, we argue 
as in the proof of Proposition~\ref{prop:crys_simpl}. If we put $D' := D_0$, 
we have the isomorphism 
\ga{34.1}{ R\Gamma((X'\cap X''_{(m)}/W)_{\rm crys}, L(\sO_{D''_{(m)}})) \xrightarrow{\cong} R\Gamma((X_{(\bullet,m)}/W)_{\rm crys}, L(\sO_{D_{(\bullet,m)}}))} 
which is $R\Gamma((X/W)_{\rm crys}, (-))$ of the resolution \eqref{20}. 
From this and the induction hypothesis one deduces the isomorphism
\begin{align}
\epsilon_1^*: 
R\Gamma((X' \cap X''/W)_{\rm crys}, \sO_{ X' \cap X''/W }) & \xrightarrow{\cong}
R\Gamma((X' \cap X''_{(\bullet)}/W)_{\rm crys}, L(\sO_{D''_{(\bullet)}}))  \label{34.2} \\ & 
\xrightarrow{\cong} R\Gamma((X_{(\bullet,\bullet)}/W)_{\rm crys}, 
L(\sO_{ D_{(\bullet,\bullet)}})). \nonumber 
\end{align}
Then one has the diagram as \eqref{29} with the triangle in the right column 
is replaced by the triangle 
\begin{align}
& R\Gamma((X_{(\bullet)}/W)_{\rm crys}, 
L(\sO_{ D_{(\bullet)}})) \label{34.3} \\ 
\to \, &
R\Gamma((X'_{(\bullet)}/W)_{\rm crys}, 
L(\sO_{ D'_{(\bullet)}})) \oplus 
R\Gamma((X''_{(\bullet)}/W)_{\rm crys}, 
L(\sO_{ D''_{(\bullet)}})) \nonumber \\ 
\to \, &
R\Gamma((X_{(\bullet,\bullet)}/W)_{\rm crys}, 
L(\sO_{ D_{(\bullet,\bullet)}})), \nonumber 
\end{align}
which exists by consrtuction. 
%In general, one has the exact triangle as on the right column of \eqref{19} with  $\sO_{X_{(\bullet)}}/W$ replaced by $L(\sO_{D_{(\bullet)}})$, and similarly for $X'$ and $X''$. 
Thus one concludes by induction.

\subsection{Proof of Theorem~\ref{thm1:main}}
One defines $G$ to be the group scheme $GL(r)$ over $\Z$, and for any scheme $S$, one writes $G_S$ for the induced group scheme over $S$, or by abuse of notations $G_R$ for  $S={\rm Spec}(R)$.  So one has $G_k, G_W$ and its $p$-adic completion $\hat G_W$. One has the classifying (formal) simplicial schemes  $BG_k, BG_W, B \hat G_W$. The datum of a locally free crystal $E$ of rank $r$, together with a trivialization
of $E_X|_{X_i}=E_{X_i}$, where $X=\cup_{i\in I} X_i$ is a finite affine covering, and a lift of the trivialization to a trivialization of $E_D|_{D_i}$, yields, via the transition functions, a commutative diagram
\ga{35}{\xymatrix{ \ar[d]_f X_{(\bullet)} \ar[r]^{\alpha}  & D_{(\bullet)} \ar[d]^g  \ar[dr]^h \\
   BG_k \ar[r]^{\hat \beta}  & B\hat G_W \ar[r]^{\iota} & BG_W
}}
where $\alpha$, $\hat  \beta$  and $\iota$ are the canonical morphisms.
The crystalline Chern classes come as the pull-back from the universal Chern classes via 
\ga{36}{ f^*: H^{2i}((BG_k/W)_{\rm crys}, \sO_{BG_k/W})\to H^{2i}((X_{(\bullet )}/W)_{\rm crys},  \sO_{X_{(\bullet)}/W})}
computed via
\ga{37}{ h^*:  H^{2i}(BG_W,  \Omega^\bullet_{BG_W})  \xrightarrow{\iota^*}  H^{2i}(B \hat G_W,  \Omega^\bullet_{B\hat G_W})\xrightarrow{g^*}  H^{2i}(D_{(\bullet )}, \bar  \Omega^\bullet_{D_{(\bullet)}/W}) \xrightarrow{\cong} 
H^{2i}(D_{(\bullet )}, \sO^\bullet_{D_{(\bullet)}/W}). 
}
Thus $h^*$ factors through  the forgetful map  $H^{2i}(BG_W,  \Omega^\bullet_{BG_W})  \to H^{2i}(BG_W,  \sO_{BG_W})$, which is zero for $i\ge 1$ 
because the composition of it with the surjective map 
$H^{2i}(BG_W,  \Omega^{\ge i}_{BG_W}) \to H^{2i}(BG_W,  \Omega^\bullet_{BG_W})$ 
(\cite[II Thm. ~1.1]{Gro90}) is zero. 
This proves Theorem~\ref{thm1:main}. 
\end{proof}

\begin{remark}\label{rem:torsion2} 
By replacing $W$ by $W_n = W/p^nW$, 
we see that the proof above 
also gives the following variant of Theorem \ref{thm1:main}: 
If $X$ is as in Theorem \ref{thm1:main} and
if $E$ is a locally free crystal on $X/W_n$, 
then $c_i^{\rm crys}(E_X)$ is zero in the torsion crystalline cohomology group 
$H^{2i}_{\rm crys}(X/W_n)$ for $i\ge 1$. 
In the final step, we need the surjectivity 
of the map $H^{2i}(BG_{W_n},  \Omega^{\ge i}_{BG_{W_n}}) \to H^{2i}(BG_{W_n},  \Omega^\bullet_{BG_{W_n}})$, but it follows from the isomorphism 
$$H^{2i}(BG_{W},  \Omega^\bullet_{BG_{W}}) \otimes_W W_n = 
H^{2i}(BG_{W_n},  \Omega^\bullet_{BG_{W_n}}) $$ which is true because 
$H^{n}(BG_{W},  \Omega^\bullet_{BG_{W}}) \, (n \in \N)$ 
are free over $W$ (\cite[II Thm. ~1.1]{Gro90}). 
\end{remark}

%\section{Crystalline Chern classes of torsion free crystals on smooth projective varieties} \label{s:Xproj}
%In this section we assume $X$ to be smooth projective over a perfect field $k$, and $E$ to be a torsion free crystal. We prove Theorem~\ref{thm1:main} i).

%Let $\iota: X\hookrightarrow \P_W$ to be a closed immersion in a projective space over ${\rm Spf}(W)$, $\sI$ be the sheaf of ideals defined by $\iota$.  We first recall Ogus' construction of the universal $p$-adic enlargement associated to $\iota$ (\cite[Prop.2.3.]{Ogu84}). Let $b:B \to \P_W$ be the $p$-adic formal blow up of $\sI$. Then $T$ is the open $p$-adic formal subscheme which is the complement in $B$ of the strict transform of $\P_k$. Equivalently, a point $x\in B$ lies in $T$ if and only if the ideal spanned by $\sI$ in $\sO_{B,x}$ is the ideal spanned by $b$. So the restriction $b_T:T\to \P_W$ of $b$ to $T$ is smooth, not proper.  One has a commutative diagram of fiber squares
%\ga{38}{\xymatrix{\ar @{} [dr] |{\qed}
%\ar[d]_j X' \ar@{^{(}->}[r]  &T \ar[d]\\
%\ar @{} [dr] |{\qed}
%\ar[d]_{b'} B'\ar[r]  & B \ar[d]^b\\
%X  \ar@{^{(}->}[r]^{\iota}  & \P_W
%}}
%As $b$ is projective, so is $b'$. More precisely, $b'$ is a $\P^c_k$-bundle over $X$, where $c$ is the codimension of $X$ in $\P_k$. 
%Then  $j$ is an affine embedding, is a $\A_k^c$-bundle over $X$.  Moreover,  $B'\setminus X'=:H$ is a relatively ample divisor in $B'$. 

\section{Generalization of the main Theorem~\cite[Thm.1.1]{ES15} from convergent isocrystals to isocrystals possessing a locally free lattice} \label{s:thm1.1}
The aim of this section is to prove Theorem~\ref{thm:dejonglocfree}.
The proof is the same as the one of \cite[Thm.1.1]{ES15} except some points 
which we explain now. 

\medskip

First assume that $k$ is algebraically closed. 
We have $H^1_{\rm crys}(X/W)=0$ by \cite[Prop. ~2.9(2)]{ES15}. 
(See also Theorem \ref{thm:van}.) So we may assume that $\sE$ is irreducible 
(of rank $s \leq r$) to prove the theorem. 
Also, there exists $N \in \N$ such that 
the restriction $H^1_{\rm crys}(X/W_n) \to H^1_{\rm crys}(X/k)$ is zero 
for any $n \ge N$. 
By using Theorem~\ref{thm1:main} (and the remark after it) in place of 
\cite[Prop.~3.1]{ES15}, we see that $\sE$ admits a lattice $E$ with 
$E_X$ strongly $\mu$-stable as $\sO_X$-module, by \cite[Prop.~4.2]{ES15}. 
Because $E_X$ has vanishing Chern classes, we can argue as in 
Section 3 in \cite{ES15}, and by \cite[Cor.~3.8]{ES15}, we see that 
there exists $a \in \N$ such that $(F^a)^*E_N \in {\rm Crys}(X/W_N)$ is trivial. 

Let $\sD_{n,m}$ be the category of pairs $(G, \varphi)$, where 
$G \in {\rm Crys}(X/W_{n+m})$ and $\varphi$ is an isomorphism 
between the restriction of $G$ to ${\rm Crys}(X/W_n)$ and $\sO_{X/W_n}^s$. 
Then the same computation as \cite[Prop.~3.6]{ES15} implies the isomorphism 
as pointed sets 
\ga{37.01}{\sD_{n,m} \cong M(s\times s, H^1_{\rm crys}(X/W_m))} 
for $1 \leq m \leq n$, which is compatible with respect to $m$.

Applying \eqref{37.01} to the pairs $(n,m)=(N, N), (N,1)$, we conclude that 
$(F^a)^*E_{N+1} \in {\rm Crys}(X/W_{N+1})$, which is the image of 
$(F^a)^*E_{2N} \in {\rm Crys}(X/W_{2N})$ via the restriction 
$$ \sD_{N,N} \cong M(s\times s, H^1_{\rm crys}(X/W_N)) \to 
M(s\times s, H^1_{\rm crys}(X/W_1)) \cong \sD_{N,1} $$
is trivial. We continue similarly to show that 
$(F^a)^*E_{n} \in {\rm Crys}(X/W_{n})$ is trivial for all $n \ge N$. 
Hence $(F^a)^*E$ is trivial. Because 
the endofunctor $F^*: {\rm Crys}(X/W) \to {\rm Crys}(X/W)$ is fully faithful 
at least when restricted to locally free crystals (\cite[Ex. 7.3.4]{Ogu94}), 
we see that $E$ itself is trivial. Hence $\sE$ is also trivial. 

Now we prove the theorem for general $k$. 
If we take a locally free lattice $E$ of $\sE$ and denote the 
pullback of $E$ to $(X \otimes_k \overline{k}/W(\overline{k}))_{\rm crys}$ by 
$\overline{E}$, we see by base change theorem that 
$\Q \otimes_{\Z} (W(\overline{k}) \otimes_{W} H^0((X/W)_{\rm crys}, E)) \cong \Q \otimes_{\Z} 
H^0((X \otimes_k \overline{k}/W(\overline{k}))_{\rm crys}, E)$, and the latter 
is $s$-dimensional over the fraction field of $W(\overline{k})$. 
So $\Q \otimes_{\Z} H^0((X/W)_{\rm crys}, E)$ is $s$-dimensional over $K$ and 
hence $\sE$ is trivial. 

%\end{proof}

\section{Gau{\ss}-Manin convergent isocrystal} \label{s:gm}
The aim of this section is to prove Theorem~\ref{thm2:main}.  The proof is inspired by the discussion with G. Faltings related in the introduction. 

\subsection{First rigid cohomology of a smooth simply connected variety is trivial} \label{ss:H1}
We shall prove the following theorem which is used in order to pass from the triviality of the semi-simplification of 
the Gau{\ss}-Manin overconvergent isocrystal to the triviality of the Gau{\ss}-Manin overconvergent isocrystal itself.

As usual, if $k$ is a perfect field, then one denotes by $W=W(k)$ its ring of Witt vectors and by $K$ the field of fractions of $W$.
\begin{thm} \label{thm:van}
Let $X$ be a smooth connected variety defined over an algebraically closed field $k$ of 
characteristic $p>0$, and assume that 
$X$ is proper or $p \geq 3$. If 
$\pi_1^{\rm \acute{e}t,ab}(X)=0$, then 
$H^1_{\rm rig}(X/K)=0$. 
\end{thm}

\begin{proof}
If $X$ is proper, this is \cite[ Thm.~0.1 1)]{ES15}. Else, 
let $\ell$ be a prime not equal to $p$. 
When $X$ is proper, the assumption $\pi_1^{\rm \acute{e}t,ab}(X)=0$ 
implies that 
$0 = H^1_{\rm \acute{e}t}(X,\Q_{\ell}) = 
H^1_{\rm \acute{e}t}({\rm Pic}^0_{\rm red}(X),\Q_{\ell})$ and so 
${\rm Pic}^0_{\rm red}(X) = 0$. Hence $H^1_{\rm crys}(X/W)=0$ because it is 
the Dieudonn\'e crystal associated to ${\rm Pic}^0_{\rm red}(X)$ and so 
$H^1_{\rm rig}(X/K)$ is also equal to $0$. 

In case $X$ is not proper, we use the Picard $1$-motive $M := {\rm Pic}^+(X)$ 
defined in \cite{ABV05}. Put $M = [L \to G]$, where $L$ is a $\Z$-module 
and $G$ is a semi-abelian variety. 
Because the $\ell$-adic realization 
$V_lM$ of $M$ is equal to $H^1_{\rm \acute{e}t}(X,\Q_{\ell})$ by \cite{Man13} and 
 fits into the exact sequence 
$$ 0 \longrightarrow V_{\ell}G \longrightarrow V_{\ell}M \longrightarrow V_{\ell}L \longrightarrow 
0, $$
we have $G = 0, L = 0$ hence $M=0$. So the crystalline realization of $M$, 
which is equal to $H^1_{\rm rig}(X/K)$ by \cite{ABV05} when $p \geq 3$, vanishes.  
\end{proof}

\subsection{Gau{\ss}-Manin convergent isocrystal}     \label{ss:GMiso}
In this subsection, we give some preliminaries on convergent $F$-isocrystals and 
then recall Ogus' definition of Gau{\ss}-Manin convergent $F$-isocrystal $R^if_*\sO_{Y/K}$ for a smooth proper morphism $f:Y\to X$ with $X$ smooth over a perfect field $k$ (see \cite[Section~3]{Ogu84}). 

\medskip

Recall that objects in the convergent site on $X$ over $K$ are  
enlargements, which are the diagrams of the form  
$(X\xleftarrow{} (T\otimes_W k)_{\rm red} \hookrightarrow T)$ 
over $W$, where $T$ is a $p$-adic formal scheme of finite type and flat over $W$. 
One defines  a convergent isocrystal on $X/K$ as a crystal of 
$\Q \otimes_{\Z} \sO_T$-modules on enlargements. Crystal means 
 a sheaf of coherent $\Q \otimes_{\Z} \sO_T$-modules  with transition functions being isomorphisms. 
We denote the category of convergent isocrystals on $X/K$ 
by ${\rm Conv}(X/K)$. 
The convergent site is functorial for $X/W$, and so we can define the
pullback functor $F^*: {\rm Conv}(X/K) \to {\rm Conv}(X/K)$ 
induced by the Frobenius $(F_X, F_W)$ on $X$ and $W$. 
Then we define the category $F\text{-}{\rm Conv}(X/K)$ of 
convergent $F$-isocrystals on $X/K$ as the category of pairs 
$(E, \Phi)$, where $E \in {\rm Conv}(X/K)$ and $\Phi$ is an isomorphism 
$F^*E \to E$. $\Phi$ is called a Frobenius structure on $E$. 

We define the $p$-adic convergent site on $X$ over $K$ 
as a variant of the convergent site: 
The objects in it are the $p$-adic enlargements, 
which are the diagrams of the form 
$(X\xleftarrow{} T\otimes_W k \hookrightarrow T)$ over $W$, where $T$ is as before. 
As in the case of convergent site, we can define the category of 
$p$-adic convergent isocrystals and that of 
$p$-adic convergent $F$-isocrystals, which we denote by 
$p{\rm Conv}(X/K)$, 
$F\text{-}p{\rm Conv}(X/K)$ respectively. 

Then we have the sequence of functors 
\ga{37.1}{ 
F\text{-}{\rm Conv}(X/K) \to F\text{-}{\rm Crys}(X/W)_{\Q} \to  
F\text{-}p{\rm Conv}(X/K), 
}
in which the first one is the inverse of the functor $M \mapsto M^{\rm an}$ in 
\cite[Thm.~2.4.2]{Ber96}. For any $p$-adic enlargement 
$T := (X\xleftarrow{h} T\otimes_W k \hookrightarrow T)$, 
$T_n := (T\otimes_W k \hookrightarrow T \otimes W/p^nW) \,(n \in \N)$ 
are objects in the crystalline site $(T \otimes k/W)_{\rm crys}$. 
Then, for any $E \in {\rm Crys}(X/W)$,  
$T \mapsto \varprojlim_n (h^*E)_{T_n}$ defines an object in 
$p{\rm Conv}(X/K)$ and this induces the second functor in \eqref{37.1}.
%{\blue  One then uses that 
%the category of crystals is unchanged 
%as one goes from the  crystalline site ito the big one. REFERENCE? This defines the second functor.
%}
The functors in \eqref{37.1} are known to be equivalences 
by \cite[Thm.~2.4.2]{Ber96} and \cite[Prop.~2.18]{Ogu84}.

\medskip

We recall the definition of the Gau{\ss}-Manin convergent $F$-isocrystal 
$$R^if_*\sO_{Y/K} \in F\text{-}{\rm Conv}(X/K)$$
for a smooth proper morphism $f:Y\to X$ (see \cite[Section~3]{Ogu84}). 
It is defined as the unique object such that, for any $p$-adic enlargement 
$T := (X\xleftarrow{} T\otimes_W k \hookrightarrow T)$, 
the value $(R^if_*\sO_{Y/K})_T$ at $T$ of $R^if_*\sO_{Y/K}$  as 
an object in $F\text{-}p{\rm Conv}(X/K)$ is given by 
$\Q\otimes R^n(f_T)_{\rm crys *} \sO_{Y_T/T}$, where 
$(f_T)_{\rm crys}: (Y_T/T)_{\rm crys} \to T_{\rm Zar}$ 
is the morphism of topo\"i induced by 
the pull-back $f_T: Y_T \to T \otimes k$ 
of $f$ by $T\otimes k \to X$. 

\medskip 

Assume now that $X$ admits a closed embedding  
$\iota: X \hookrightarrow P$ into a $p$-adic smooth formal scheme $P$ over ${\rm Spf}\,W$ and give 
another description of convergent $F$-isocrystals and the Gau{\ss}-Manin convergent isocrystals. 
For $n = 0,1,2$, let $X \leftarrow X(n) \hookrightarrow \sZ(n)$ be the universal $p$-adic enlargement \cite[Prop.~2.3]{Ogu84}
%{\blue please help here: enlargements are triples...and what is universal here? references}
of $X \hookrightarrow \underbrace{P \times_W \cdots \times_W P}_{n+1}$, and let 
$F\text{-}{\rm Str}(X \hookrightarrow P/W)$ be the category of triples 
$(E,\epsilon,\Phi)$ consisting of 
a coherent $\Q \otimes_{\Z} \sO_{\sZ(0)}$-module $E$, an isomorphism 
$\epsilon: p_2^*E \overset{\cong}{\to} p_1^*E$ ($p_i: \sZ(1) \to\sZ(0)$ are 
the projections) satisfying cocycle condition, and an isomorphism 
$\Phi: F^*(E,\epsilon) \overset{\cong}{\to} (E,\epsilon)$ 
(where $F$ is the pullback by Frobenius on $X(n) \hookrightarrow \sZ(n)$). 
Then we have an equivalence of categories [(2.11), Ogus]
\ga{37.2}{
F\text{-}p{\rm Conv}(X/K) \cong F\text{-}{\rm Str}(X \hookrightarrow P/W). }

Via the equivalences \eqref{37.1} and \eqref{37.2}, the Gau{\ss}-Manin 
convergent isocrystal $R^if_*\sO_{Y/K}$ is described as the triple  
$(\Q \otimes R^nf(0)_{{\rm crys} *}\sO_{Y(0)/\sZ(0)}, \epsilon, \Phi) \in F\text{-}{\rm Str}(X \hookrightarrow P/W)$, where 
$f(0)_{\rm crys}: (Y(0)/\sZ(0))_{\rm crys} \to \sZ(0)_{\rm Zar}$ 
is the morphism 
of topoi induced by 
the pull-back $f(0): Y(0) \to X(0)$ 
of $f$ by $X(0) \to X$ and 
$\epsilon, \Phi$ are defined by the functoriality of 
crystalline cohomology sheaves.

\medskip

Let us denote the category of overconvergent $F$-isocrystals on $X/K$ 
\cite[(2.3.7)]{Ber96} by $F\text{-}{\rm Conv}^{\dagger}(X/K)$, and 
the structure overconvergent isocrystal by $\sO_{X/K}^{\dagger}$ to 
distinguish it from $\sO_{X/K}$. 
There exists a natural resriction functor 
$F\text{-}{\rm Conv}^{\dagger}(X/K) \to F\text{-}{\rm Conv}(X/K)$ which 
is fully faithful (\cite{Ked04}). 
When $f: Y \to X$ is smooth projective,  
Lazda \cite[Cor.~5.4]{Laz15} proved that the convergent $F$-isocrystal 
$R^if_* \sO_{Y/K}$ lifts to an overconvergent $F$-isocrystal on $X$, 
which we denote by $R^if_* \sO_{Y/K}^{\dagger}$. 
When $X$ is proper, the categories $F\text{-}{\rm Conv}(X/K)$, 
$F\text{-}{\rm Conv}^{\dagger}(X/K)$ are the same 
% restriction functor is an equivalence
and so one can equate $R^if_* \sO_{Y/K}^{\dagger}$ and $R^if_* \sO_{Y/K}$. 

\begin{remark}\label{rem:frobd}
In this subsection, the Frobenius structure on a(n) (over)convergent isocrystal 
is defined with respect to the pullback functor induced by 
the Frobenius $(F_X, F_W)$ on $X$ and $W$. For any $d \ge 1$,  
we can also define the Frobenius structure 
with respect to the pullback functor induced by $(F_X^d, F_W^d)$, 
and the Frobenius structure in the former case induces the one in the latter case. 
Such a Frobenius structure will appear in the next subsection. 
\end{remark} 

\begin{comment}
We first recall the definition of the convergent isocrystal $R^if_*\sO_{Y/K}$ for $f:Y\to X$ a  smooth proper  morphism of smooth varieties defined over a perfect field $k$ (see \cite[Section~3]{Ogu84}). To a $p$-adic enlargement $\epsilon: (X\xleftarrow{} T\otimes_W k \hookrightarrow T)$, one associates the pull-back morphism $f_T: Y\times_X T\to T$.  Then the value of $R^if_*\sO_{Y/K}$ on $\epsilon$ is 
$\Q\otimes R^n(f_T)_{\rm crys *} \sO_{Y_T/T} \in {\rm Crys}(X/W)_{\Q}$, where $R^n(f_T)_{\rm crys *} \sO_{Y_T/T} \in {\rm Crys}(X/W)_{\Q}$. See \cite[Defn.3.1]{Ogu84} and \cite[(7.24)]{BO78}.

In fact there is more structure on $R^if_*\sO_{Y/K}$. The Frobenius acts on it, thus this is a $F$-convergent isocrystal \cite[Thm.3.7]{Ogu84}, but we shall not need this in a first step. 
\end{comment} 

\subsection{ Case where the ground field $k$ is finite.}
In this subsection, we prove Theorem~\ref{thm2:main} 
in the case where $k$ is a finite field
and the statement in Remark 
\ref{rem:thm2}. 
In order to prove them, we may replace $k$ by a finite extension $k'$, and  $K$ by a finite possibly ramified extension of the field of fractions of the ring of Witt vectors over $k'$, 
such that the following hold:
\begin{enumerate}
\item $X$ has a $k$-rational point $x$. 
\item The eigenvalues of the action of Frobenius $F_x$ on 
$L_{{\ell}} := H^i_{\rm \acute{e}t}(Y \times_X \bar{x}, \Q_{\ell})$ belong to a 
number field $K_0$ contained in $K$. 
\item There exists a $K_0$-vector space $L$ with linear action $F$ and 
an inclusion of fields $K_0 \hookrightarrow \Q_{\ell}$ such that 
$L \otimes_{K_0} \Q_{\ell}$ is isomorphic to $L_{\ell}$ as vector spaces 
with an action. 
\end{enumerate}
 (Note that overconvergent isocrystals and rigid cohomologies are 
defined even when the base complete discrete valuation ring 
$O_K$ is ramified over $W(k)$ (\cite{Ber96}, \cite{LeS07}), and the rigid cohomology 
satisfies the base change property for finite extension of the base (\cite[Cor.~11.8.2]{ChT03}).) 

Until the end of this subsection, we consider the Frobenius structure 
with respect to the pullback induced by $(F_X^d, {\rm id})$, where 
$d = \log_p |k|$. (See Remark \ref{rem:frobd}.) 
For any closed point $y$ in $X$, one has the base change isomorphism 
\ga{39}{ y^* R^if_*\sO_{Y/K}^{\dagger} = 
y^* R^if_*\sO_{Y/K}=H^i_{\rm rig} (Y\times_X y/K)}
by \cite[Rmk.3.7.1]{Ogu84} and \cite[Prop.~1.8, Prop.~1.9]{Ber97}, and the set of Frobenius eigenvalues on it is 
the same as that on $\ell$-adic cohomology 
$H^i (Y\times_X \bar y, \Q_\ell)$ 
by \cite[Thm.1]{KM74}, \cite[Cor.~1.3]{CLS98}. 
The assumption $\pi_1^{\rm \acute{e}t}(X \otimes \overline k)=\{1\}$ implies that 
the action of Frobenius $F_y$ on $H^i (Y\times_X \bar y, \Q_\ell)$ is 
identified with the action of $F_x^{d_y}$ (where $d_y={\rm deg}(y/k)$) on $L_{\ell}$, 
hence the action of $F \otimes {\rm id}$ on $L \otimes_{K_0} \Q_{\ell}$. 

Let $\sE_0$ be the overconvergent $F$-isocrystal on $X$ defined by 
$((L \otimes_{K_0} K) \otimes_K \sO_{X/K}^{\dagger}, F \otimes {\rm id})$. 
Then, by construction, $R^if_*\sO_{Y/K}^{\dagger}$ and 
$\sE_0$ have the same eigenvalues of Frobenius action on any closed point of 
$X$. Then Abe's \v{C}ebotarev's density theorem  \cite[A.3]{Abe13} implies that the semi-simplification of $R^if_*\sO_{Y/K}^{\dagger}$ is the same as that of 
$\sE_0$. Hence the semi-simplification of $R^if_*\sO_{Y/K}^{\dagger}$ 
is trivial as overconvergent isocrystal on $X$. 
By Theorem~\ref{thm:van},  extensions of $\Q\otimes \sO^{\dagger}_{X/W}$ by itself 
are trivial when $X$ is proper or $p \geq 3$. This finishes the proof. 

%we may replace $X$ by $X\otimes_K k'$ where $k\hookrightarrow k'$ is a finite  field extension. In particular, we may assume that $X$ has a rational point, say $x\in X(k)$. Let $\epsilon: X\to {\rm Spec}(k)$ be the structure morphism.  Then we wish to show 
%\ga{  38}{ R^if_*\sO_{Y/K}= \epsilon^* x^* R^if_*\sO_{Y/K}.}
%By  base change \cite[Rmk.3.7.1.]{Ogu84}, for any closed point $y\in |X|$, one has
%\ga{39}{ y^* R^if_*\sO_{Y/K}=H^i_{\rm crys} (Y\times_X y, \sO_{Y\times_X y/K}),}
%and by \cite[ Thm.1 ]{KM74}, the action of the Frobenius $F_y$  of $k(y)$ on 
%$H^i_{\rm crys} (Y\times_X y, \sO_{Y\times_X y/K})$ is the same as its action on $\ell$-adic cohomology 
%$H^i (Y\times_X \bar y, \Q_\ell)$, where $\bar y\to y$ is a $\bar \F_p$-point above $y$. 
%On the other hand, the assumption $\pi_1^{\rm \acute{e}t}(X)=\{1\}$ implies that this action is constant in the sense that if $d={\rm deg}(y/k)$, then $F_y$ acts as $F_x^d$.

\subsection{General case} 
In this subsection, we prove Theorem~\ref{thm2:main} by a
spreading out argument allowing $k$ to be a perfect field, but assuming $X$ to be proper. 
Let $f: Y \to X$ be as  in the statement of 
Theorem~\ref{thm2:main} and let $g: X \to {\rm Spec}\,k$ be the structure 
morphism. Also, let $X = \bigcup_{i \in I} X_i$ be an affine open covering of $X$ 
and take a closed embedding $X_i \to P_i$ into a smooth $p$-adic formal scheme $P_i$ 
over $W$ for each $i \in I$. We prove that the Gau{\ss}-Manin convergent isocrystal 
$\sE_{\rm conv} := R^if_*\sO_{Y/K} \in F\text{-}{\rm Conv}(X/K)$ is trivial as 
an object in ${\rm Conv}(X/K)$. 
Denote by $\sE_{\rm crys}, \sE_{p{\rm conv}}$ the image of $\sE_{\rm conv}$ 
in $F\text{-}{\rm Crys}(X/W)_{\Q}, F\text{-}p{\rm Conv}(X/K)$ via \eqref{37.1}.

We can find a connected affine scheme $T = {\rm Spec}\,A_1$ 
smooth of finite type over $\F_p$, a $p$-adic formal lift $\sT := {\rm Spf}\,A$ of $T$ 
which is smooth over ${\rm Spf}\,\Z_p$ and endowed with a lift of Frobenius,  
and proper smooth morphisms $Y_T \overset{f_T}{\to} X_T \overset{g_T}{\to} T$
which fit into the commutative diagram 
\begin{equation}\label{40} 
\xymatrix{
Y \ar[d]_g \ar[r] \ar@{}[rd]|\Box & Y_T 
\ar[d]_{f_T} \ar@{=}[r] & Y_T \ar[d]_{f_T} 
\\ X \ar[d]_g \ar[r]_{\alpha} \ar@{}[rd]|\Box & 
X_T \ar[d]_{g_T} \ar@{=}[r] & X_T \ar[d] \\ 
{\rm Spec}\,k \ar[d] \ar[r] \ar@{}[rd]|\Box 
& T \ar[d] \ar[r] \ar@{}[rd]|\Box & {\rm Spec}\,\F_p \ar[d] \\ 
{\rm Spf}\,W \ar[r] & \sT \ar[r] & {\rm Spf}\,\Z_p, 
}
\end{equation}
where the squares with symbol $\Box$ means cartesian squares. 
Also, we may assume the existence of an open covering 
$X_T = \bigcup_{i \in I} X_{i,T}$ which induces $X = \bigcup_{i \in I} X_i$ and 
closed immersions 
$X_{i,T} \hookrightarrow P_{i,T} \, (i \in I)$ into a $p$-adic formal scheme $P_{i,T}$ smooth 
over $\sT$ with $P_{i,T} \times_{\sT} {\rm Spf}\,W = P_i$. 

Let $\widetilde{\sE}_{T, {\rm conv}} := R^ng_{T,{\rm conv},*}\sO_{Y_T/\Q_p} 
 \in F\text{-}{\rm Conv}(X_T/\Q_p)$ be the Gau{\ss}-Manin convergent 
$F$-isocrystal defined by the right column of \eqref{40}, and let 
$\widetilde{\sE}^n_{T, {\rm crys}}$, 
$\widetilde{\sE}^n_{T, p{\rm conv}}$ 
be its images in the categories $F\text{-}{\rm Crys}(X_T/\Z_p)_{\Q}, 
F\text{-}p{\rm Conv}(X_T/\Q_p)$. Then the morphism $\alpha$ 
in \eqref{40} induces 
the pullback functor 
\ga{40.1}{F\text{-}{\rm Crys}(X_T/\Z_p)_{\Q} \to F\text{-}{\rm Crys}(X_T/\sT)_{\Q} 
\overset{\alpha^*}{\to} F\text{-}{\rm Crys}(X/\Z_p)_{\Q}.}  %
%& F\text{-}p{\rm Conv}(X_T/\Q_p) \cong 
%F\text{-}{\rm Str}(X_T \hookrightarrow \P_T/\Z_p)  
%\overset{\alpha^*}{\to} 
%F\text{-}{\rm Str}(X \hookrightarrow \P_W/W) \cong F\text{-}p{\rm Conv}(X/\Q_p) \nonumber 
%which are equal via the second equivalence in \eqref{37.1}. 
We denote this functor by $\widetilde{\alpha}^*$. Then 
we have the map
\ga{41}{ 
\widetilde{\alpha}^*\widetilde{\sE}_{T, {\rm crys}} \to \sE_{\rm crys}} 
%\widetilde{\alpha}^*\widetilde{\sE}_{T, p{\rm conv}} \to \sE_{p{\rm conv}}
defined by functoriality. 
%, which are the same 
%via the second equivalence in \eqref{37.1}. 

To prove that the map \eqref{41} is an isomorphism, 
we may work locally on each $X_i$. So we may assume that 
$X_T$ admits a closed embedding $X_T \to P_T$ into a $p$-adic formal 
scheme $P_T$ smooth over $\sT$. Put $P := P_T \times_{\sT} W$. 
Then the functor \eqref{40.1} is identified with the composite 
\ga{41.1}{F\text{-}p{\rm Conv}(X_T/\Q_p) \cong 
F\text{-}{\rm Str}(X_T \hookrightarrow P_T/\Z_p)  
\overset{\alpha^*}{\to} 
F\text{-}{\rm Str}(X \hookrightarrow P/W) \cong F\text{-}p{\rm Conv}(X/\Q_p)}
via the second equivalence in \eqref{37.1}. 
It suffices to prove that the map 
\ga{41.2}{\widetilde{\alpha}^*\widetilde{\sE}_{T, p{\rm conv}} \to \sE_{p{\rm conv}}}
induced by \eqref{41} via the second equivalence in \eqref{37.1} is an 
isomorphism, and to see this, it suffices to prove that the base change map 
$$ \alpha_{\sZ(n)}^* (\Q \otimes R^ig(n)_{T,{\rm crys},*}
\sO_{Y_T(n)/\sZ_T(n)}) \to  
\Q \otimes R^ig(n)_{{\rm crys},*}\sO_{Y(n)/\sZ(n)}$$ 
induced by diagrams 
\begin{equation}\label{42} 
\xymatrix{
Y(n) \ar[d]_{g(n)} \ar[r] \ar@{}[rd]|\Box & Y_T(n) \ar[d]^{g(n)_T} \\ 
X(n) \ar[d] \ar[r]  & 
X_T(n) \ar[d] \\ 
\sZ(n) \ar[r]_{\alpha_{\sZ(n)}} & \sZ_T(n)}
\end{equation}
is an isomorphism, where $X_T \leftarrow X_T(n) \hookrightarrow \sZ_T(n)$ is the universal 
$p$-adic enlargement of $X_T \hookrightarrow 
\underbrace{P_{T} \times_{\Z_p} \cdots \times_{\Z_p} P_{T}}_{n+1}$ and 
$X \leftarrow X(n) \hookrightarrow \sZ(n)$ is the universal 
$p$-adic enlargement of $X \hookrightarrow 
\underbrace{P \times_{W} \cdots \times_{W} P}_{n+1}$. 
This follows from the base change theorem of crystalline cohomology 
$$ L\alpha_{\sZ(n)}^* Rg(n)_{T,{\rm crys},*}
\sO_{Y_T(n)/\sZ_T(n)} \overset{\cong}{\to}   
Rg(n)_{{\rm crys},*}\sO_{Y(n)/\sZ(n)}$$ 
(\cite[Thm.~7.8]{BO78}, \cite[Cor.~3.2]{Ogu84}) and the flatness of 
$\Q \otimes R^jg(n)_{T,{\rm crys},*}\sO_{Y_T(n)/\sZ_T(n)}$ for $j \geq 0$ 
\cite[Cor.~2.9]{Ogu84}.  
Hence the map \eqref{41} is an isomorphism.

Let $\sE_{T, {\rm crys}}$ be the image of $\widetilde{\sE}_{T, {\rm crys}}$ by 
$F\text{-}{\rm Crys}(X_T/\Z_p)_{\Q} \to F\text{-}{\rm Crys}(X_T/\sT)_{\Q}$. 
Then the map \eqref{41} can be rewritten as 
\begin{equation}\label{43} 
\alpha^*\sE_{T, {\rm crys}} \to \sE_{\rm crys}. 
\end{equation}
Hence the map \eqref{43} is also an isomorphism. 

Next, let $s = {\rm Spec}\,\F_q$ be a closed point of $T$ and put 
$\sS = {\rm Spf}\,W(\F_q) =: {\rm Spf}\,\Z_q, \Q_q := \Q \otimes \Z_q$. 
Then, by taking the fiber at $s$, we obtain the following diagram: 
\begin{equation}\label{44} 
\xymatrix{
Y \ar[d]_g \ar[r] \ar@{}[rd]|\Box & Y_T 
\ar[d]_{f_T} \ar@{}[rd]|\Box & Y_s \ar[l] \ar[d]_{f_s} 
\\ X \ar[d]_f \ar[r]_{\alpha} \ar@{}[rd]|\Box & 
X_T \ar[d]_{g_T} \ar@{}[rd]|\Box & X_s \ar[l]^{\beta} \ar[d]_{g_s} \\ 
{\rm Spec}\,k \ar[d] \ar[r] \ar@{}[rd]|\Box 
& T \ar[d] \ar@{}[rd]|\Box & s \ar[l] \ar[d] \\ 
{\rm Spf}\,W \ar[r] & \sT & \sS. \ar[l] 
}
\end{equation}
Let $\sE_{s, {\rm conv}} := R^ig_{s,{\rm conv},*}\sO_{Y_s/\Q_q} 
 \in F\text{-}{\rm Conv}(X_s/\Q_q)$ be the Gau{\ss}-Manin convergent 
$F$-isocrystal defined by the right column of \eqref{44}, and 
denote its image in $F\text{-}{\rm Crys}(X_s/\Z_q)_{\Q}$ by 
$\sE_{s, {\rm crys}}$. 
By the same method as above, one can prove 
the isomorphism 
\begin{equation}\label{45} 
\beta^*\sE^n_{T, {\rm crys}} \to \sE^n_{s,{\rm crys}}, 
\end{equation}
where $\beta^*$ is the pullback 
$F\text{-}{\rm Crys}(X_T/\sT)_{\Q} \to F\text{-}{\rm Crys}(X_s/\Z_q)_{\Q}$ 
by $\beta$. 

We consider the crystalline cohomology 
$H^j((X_T/A)_{\rm crys}, \sE_{T,{\rm crys}}) \, (j \geq 0)$. This is a finitely generated 
$\Q \otimes A$-module as we will prove below. 
Since $A_{pA}$ is a discrete valuation ring, we may assume that 
$H^j((X_T/A)_{\rm crys}, \sE_{T,{\rm crys}})$'s are free 
$\Q \otimes A$-modules, by shrinking $\sT = {\rm Spf}\,A$. 
Then we have the base change isomorphisms 
\begin{align}
\Q_q \otimes_{\Q \otimes A} H^j((X_T/A)_{\rm crys}, \sE_{T,{\rm crys}})
\cong H^j((X_s/\Z_q)_{\rm crys}, \sE^n_{s,{\rm crys}}), \label{46} \\ 
K \otimes_{\Q \otimes A} H^i((X_T/A)_{\rm crys}, \sE_{T,{\rm crys}})
\cong H^j((X/W)_{\rm crys}, \sE_{\rm crys}), \label{47} 
\end{align}
which we will prove below. 

Let $r$ be the rank of $\sE_{T, {\rm crys}}$.
By the triviality of Gau{\ss}-Manin convergent isocrystal 
in finite field case, $\sE_{s,{\rm crys}}$ is trivial 
as a convergent isocrystal. Hence, by \eqref{46}, 
$H^0((X_T/A)_{\rm crys}, \sE_{T,{\rm crys}}) \allowbreak \cong (\Q \otimes A)^r$ 
and by \eqref{47}, 
$H^0((X/W)_{\rm crys}, \sE_{\rm crys}) \cong K^r$. 
Hence $\sE_{\rm crys}$ is trivial (and so $\sE_{\rm conv}$ is also trivial), 
as required. So the proof of the theorem is finished modulo 
the finiteness and the base change property we used above. 

Finally, we prove the finiteness and the base change property. 
%(This is essentially written in \cite[\S 1]{Shiho07}, but we reproduce it.) 
%We change the notation as follows: 
In the following, let ${\rm Spec}\,A_1$ be an affine regular scheme of 
characteristic $p>0$, let 
${\rm Spf}\,A$ be a $p$-adic formal scheme flat over ${\rm Spf}\,\Z_p$ such that 
$A/pA = A_1$. Let $X \to {\rm Spec}\,A_1$ be a projective smooth morphism,  
\begin{comment}
- $X \hra \Pr_{A}$ is a fixed closed immersion into a projective space, 
and let $X \hra D$ be its PD-envelope. \\ 
- $A_n := A/p^nA, D_n := D \otimes_{\Z_p} \Z/p^n\Z$. \\ 
\end{comment}
and let $\sE$ be an isocrystal on $(X/A)_{\rm crys}$ which is  
locally free as an isocrystal. 
%Assume also that $\sE = \Q \otimes E$ for a 
%$p$-torsion free crystal $E$ on $(X/A)_{\rm \crys}$. 
%(Note that $E$ need not be locally free. This is the reason that 
%the base change theorem in [BO] is not enough.) 
Moreover, assume that we are given the following cartesian diagram 
such that ${\rm Spec}\,A'_1$ is also regular. 
\begin{equation*}
\xymatrix{
X \ar[d] \ar@{}[rd]|\Box & 
X' \ar[l]_{\alpha} \ar[d] \\ 
{\rm Spec}\,A_1 \ar[d] \ar@{}[rd]|\Box 
& {\rm Spec}\,A'_1 \ar[l] \ar[d] \\ 
{\rm Spf}\,A & {\rm Spf}\,A'. \ar[l] 
}
\end{equation*}

Then we have the following: 

\begin{prop}
In  the situation above, we have the following: \\ 
(1) \, $R\Gamma((X/A)_{\rm crys}, \sE)$ is a perfect complex of 
$(\Q \otimes A)$-modules. \\ 
(2) \, $A' \otimes^L_A R\Gamma((X/A)_{\rm crys}, \sE) \to  
R\Gamma((X'/A')_{\rm crys}, \alpha^*\sE)$ is a quasi-isomorphism. 
\end{prop}

\begin{remark}
Although $\sE$ is locally free as an isocrystal, it is not clear 
at all that there exists a lattice $E$ of $\sE$ with $E$ locally free 
as a crystal. This is the reason that 
the base change theorem in \cite{BO78} is not enough for us. 
\end{remark} 

\begin{proof}
Let $X_{(\bullet)} \to X$ be a simplicial scheme augmented to $X$ defined 
as the \v{C}ech hypercovering associated to an open covering $X= \bigcup_{i \in I}X_i$ 
which admits a closed embedding $X_{(\bullet)} \to Y_{(\bullet)}$ into 
a simplicial $p$-adic formal scheme $Y_{(\bullet)}$ smooth over $A$, 
and let $D_{(\bullet)}$ be the PD-envelope of it. 
%Let $X \hookrightarrow \P_{A}$ be a closed immersion into a projective space 
%and let $X \hookrightarrow D$ be its PD-envelope. 
For $n \in \N$, put 
$A_n := A/p^nA, D_{(\bullet) n} := D_{(\bullet)} \otimes_{\Z_p} \Z/p^n\Z$.  
Also, let $E$ be a $p$-torsion free crystal on $X/A$ with $\sE = \Q \otimes E$, 
let $E_n$ be 
the restriction of $E$ to $(X/A_n)_{\rm crys}$ and $E_{(\bullet) n}$ the value of $E$ 
at $D_{(\bullet) n}$. 
Then 
$R\Gamma((X/A_n)_{\rm crys}, E_n) = R\Gamma(D_{(\bullet) n}, E_{(\bullet) n} \otimes \bar \Omega^{\bullet}_{D_{(\bullet) n}})$ 
and this is quasi-isomorphic to a bounded complex of 
$A_n$-modules flat over $\Z/p^n\Z$ which is compatible with respect to $n$. 
Hence 
\begin{align*}
R\Gamma((X/A_n)_{\rm crys}, E_n) \otimes^L_{A_n} A_{n-1} 
& = R\Gamma((X/A_n)_{\rm crys}, E_n) \otimes^L_{\Z/p^n\Z} \Z/p^{n-1}\Z \\ 
& = R\Gamma((X/A_{n-1})_{\rm crys}, E_{n-1}). \nonumber 
\end{align*}
Also, $R\Gamma((X/A_1)_{\rm crys}, E_1) = 
R\Gamma(X, E_X \otimes \Omega^{\bullet}_{X})$ is 
a perfect complex of $A_1$-modules, because $A_1$ is regular. 
Hence, by \cite[B.10]{BO78}, $R\Gamma((X/A)_{\rm crys}, E) 
= R\varprojlim_n R\Gamma((X/A_n)_{\rm crys}, E_n)$ 
is a perfect complex of $A$-modules, and so 
$R\Gamma((X/A)_{\rm crys}, \sE) := \Q \otimes 
R\Gamma((X/A)_{\rm crys}, E)$ is a perfect complex of 
$(\Q \otimes A)$-modules. 
This proves (1). \par 
Next we prove (2). the left hand side is equal to 
\begin{equation*}
\Q \otimes (A' \otimes^L_A R\Gamma((X/A)_{\rm crys}, E)) 
= \Q \otimes (R\varprojlim_n (A'_n \otimes^L_{A_n} 
R\Gamma((X/A_n)_{\rm crys}, E_n))
\end{equation*} 
because $R\Gamma((X/A)_{\rm crys}, E)$ is perfect, and 
the right hand side is equal to 
$$ \Q \otimes R\varprojlim_n R\Gamma 
((X/A'_n)_{\rm crys}, \alpha^*E_n). $$
Let us define $C_n$ by 
$$ C_n := {\rm Cone}(A'_n \otimes^L_{A_n} 
R\Gamma((X/A_n)_{\rm crys}, E_n) \to R\Gamma 
((X/A'_n)_{\rm crys}, \alpha^*E_n)). $$ 
Then it suffices to prove that $\Q \otimes R\varprojlim_n C_n = 0$. 
To prove this, we may work Zariski locally on $X$ (because of 
cohomological descent). So we can reduce to the case 
that $X \to {\rm Spec}\,A_1$ is liftable to an affine smooth morphism 
$\sX = {\rm Spf}\,B \to {\rm Spf}\,A$. (But we lose the projectivity of $X$.) 
Let $\sX_n := \sX \otimes_{\Z_p} \Z/p^n\Z, \sX'_n := \sX \otimes_{A_n} A'_n$ 
and let 
$E_n^{\sX}$ be the value of $E_n$ at $\sX_n$. Then 
$$ C_n = {\rm Cone}(A'_n \otimes^L_{A_n} 
\Gamma(\sX_n, E_n^{\sX} \otimes \Omega^{\bullet}_{\sX_n/A_n}) 
\to \Gamma(\sX'_n, \alpha^*E_n^{\sX} \otimes \Omega^{\bullet}_{\sX'_n/A'_n})). $$
For a finitely generated $p$-torsion free $B$-module $M$, put 
$M_n := M \otimes_A A_n$ and 
$$ C_n(M) :=  {\rm Cone}(A'_n \otimes^L_{A_n} M_n 
\to A'_n \otimes_{A_n} M_n). $$ 
Then it suffices to prove that $\Q \otimes R\varprojlim_n C_n(M) = 0$ 
when $\Q \otimes M$ is a projective $\Q \otimes B$-module. 
If we take a resolution $N^{\bullet} \to M$ of $M$ by finitely generated 
free $B$-modules, $N^{\bullet}_n \to M_n$ is also a resolution 
because $M$ is $p$-torsion free. Then 
\begin{align} 
& \Q \otimes R\varprojlim_n C_n(M) \\ 
= \, & 
\Q \otimes R\varprojlim_n{\rm Cone}(A'_n \otimes_{A_n} N^{\bullet}_n 
\to A'_n \otimes_{A_n} M_n) \nonumber \\ 
= \, & 
{\rm Cone}((\Q \otimes A') \otimes_{\Q \otimes A} (\Q \otimes N^{\bullet}) 
\to (\Q \otimes A') \otimes_{\Q \otimes A} (\Q \otimes M)) 
= 0. \nonumber 
\end{align}
(The last equality follows from the projectivity of 
$\Q \otimes N^{\bullet}$, 
$\Q \otimes M$.) 
So the proof of (2) is finished. 
\end{proof} 

\section{A very weak form of a Lefschetz theorem for isocrystals} \label{s:lefschetz}
Let $X$ be a smooth projective variety defined over a perfect field $k$ of characteristic $p>0$. 
One expects that given an irreducible isocrystal $\sE\in 
{\rm Crys}(X/W)_{\Q}$, there is an ample divisor $Y\subset X$ defined over the same field such that $\sE|_Y$ is irreducible in ${\rm Crys}(Y/W)_{\Q}$. 
 One could also expect the existence of such $Y$ which is independent of $\sE$.
One could also weaken all those variants by requesting the irreducibility only for convergent isocrystals $\sE\in {\rm Conv}(X/K)$.  Note that any subisocrystal of a convergent isocrystal is a convergent isocrystal, 
thus this version for convergent isocrystals  is just the restriction to ${\rm Conv}(X/K)$ of version for isocrystals.

We prove here a very weak form of this expectation. 
\begin{thm} \label{thm:lefschetz}
Let $X$ be a smooth projective variety defined over a perfect field $k$ of characteristic $p>0$ which lifts to $W_2(k)$. For any ample line bundle $\sH$, there is an integer $n_0>0$ such that for any $n\ge n_0$, any isocrystal $\sE\in {\rm Crys}(X/W)_{\Q}$ of rank $\le p$ which admits a lattice $E$ with $E_X$ $\mu$-stable in ${\rm Crys}(X/k)$ has the property 
that $E_Y \in {\rm Crys}(Y/k)$ is $\mu$-stable as well for any smooth divisor $Y$ in $|\sH^n|$ lifting to $W_2(k)$ on which the restriction $E_Y$ of the crystal $E_X$ to $Y$ is torsion free. Moreover, the torsion freeness condition for $E_Y$ is automatically satisfied when $\sE$ admits a locally free lattice (which is not necessarily equal to $E$). 

\end{thm}
To link Theorem~\ref{thm:lefschetz} to the expectation, one needs the following theorem.
\begin{theorem} \label{thm:lattice}
Let $X$ be a smooth projective variety defined over a perfect field $k$ of characteristic $p>0$.  Let $\sE\in {\rm Crys}(X/W)_{\Q}$ which admits a lattice $E$ such that $E_X$ is $\mu$-stable in ${\rm Crys}(X/k)$. Then $\sE$ is irreducible.

\end{theorem}

\begin{proof}
Let $\sE'\subset \sE$ be an irreducible subisocrystal and $E'$ be a lattice of $\sE'$ 
which is $\mu$-semistable in ${\rm Crys}(X/k)$. (See \cite[Prop.~4.1]{ES15} for the existence of such lattice.) 
Then there  is an integer $n$ such that $E'(nX)\subset E$,  and such that this inclusion is maximal for this property.  
By replacing $E'$ by $E'(nX)$, we can assume that the inclusion $\iota: E'\to E$ in ${\rm Crys}(X/W)$ induces 
a nonzero map $\iota_X: E'_X\to E_X$ in ${\rm Crys}(X/k)$.  
The locus $\Sigma\subset X$ on which $E'_X, E_X$ are possibly not locally free 
have codimension $\ge 2$. 
Let $\sH$ be the ample line bundle defining the $\mu$-(semi)stability, and $C$ be a smooth complete intersection curve in a high power of the linear system $|\sH|$, which is disjoint of $\Sigma$. Then the restrictions $E'_C, E_C$ of $E'_X, E_X$ to $C$, 
which is the same as the values of $E', E$ to $C$,  are locally free. 
By Theorem~\ref{thm1:main}, the degree of $E'_C, E_C$ are $0$. Thus the slopes $\mu(E'_X), \mu(E_X)$ are $0$. 
Then $0 = \mu(E'_X) \le \mu({\rm Im}\,\iota_X) < \mu(E_X) = 0$
unless $\iota_X$ is surjective. 
Hence $\iota_X$ is surjective and so is $\iota$. Hence $\sE' = \sE$. 
\end{proof}

\begin{remark} \label{rmk:irred}
The condition in Theorem~\ref{thm:lattice} implies irreducibility, but is not equivalent to it. Let $X$ be a smooth projective curve of genus $\ge 2$, with lift $X_W$ to $W=W(k)$,  let 
$a, b\in H^0(X_W,\Omega^1_{X_W/W})$ be linearly independent differential forms. Define the connection on  $E=\sO_{X_W} e_1\oplus \sO_{X_W} e_2$
by the matrix $A=\begin{pmatrix}  0 & p a\\ p b & 0 \end{pmatrix}$, which is the value of $\nabla$ on the two basis vectors $e_i$. Then $E_X$  is a trivial connection.  As a consequence, $E$ is locally nilpotent, thus is a crystal. It is irreducible as an isocrystal, as any sub line bundle of degree $0$ of $\sO_{X_K} e_1 \oplus \sO_{X_K} e_2$ is isomorphic to $\sO_{X_K}$, spanned by $\lambda_1e_1+\lambda_2e_2, \ (\lambda_1, \lambda_2)\neq (0,0)$ where $\lambda_i\in K$ and there is no form $\omega \in H^0(X_{K}, \omega_{X_K})$ such that $\lambda_2 pa=\lambda_1 \omega, \ \lambda_1 pb =\lambda_2 \omega$. 

We prove now  that there is no lattice $E'$ of $\sE$ such that 
$E'_X$ is $\mu$-stable in ${\rm Crys}(X/k)$. 
The argument is essentially the same as in \cite[Thm.~5.2]{Lan14}.
Assuming there is  such a lattice 
$E'$, after replacing $E'$ by $E'(nX)$ for a suitable integer $n$, 
there is an inclusion $\iota: E' \to E$ inducing the identity on $\sE$ such 
that the induced map $\iota_X: E'_X \to E_X$ is non-zero. 
By Theorem~\ref{thm1:main}, $\mu(E_X)=\mu(E'_X)=0$.
If $\iota_X$ is not surjective, $\mu(E'_X) < \mu({\rm Im} \iota_X) \leq \mu(E_X)$, 
a contradiction. So $\iota_X$, and thus $\iota$  are surjective.
Then $E'=E$ and thus $E_X=E'_X$, which is  impossible as $E_X$ is not stable. 

\end{remark}

\begin{remark} \label{rmk:irred2}
Irreducible convergent  isocrystals  do not necessarily satisfy 
the condition in Theorem~\ref{thm:lattice}. We now construct an example. 

Let $k$ be an algebraically closed field of characteristic $p>0$ and let 
$X$ be a smooth projective curve such that the $p$-rank of its Jacobian is $r \ge 2$. 
Let $K, K_0$ be the fraction field of $W(k)$, $W(\F_{p^2})$ respectively. 
Let $\rho: \pi_1(X) \to GL_2(\Z_p)$ be a continuous representation 
such that  $\rho_{K_0}: \pi_1(X) \to GL_2(\Z_p) \to GL_2(K_0)$ is irreducible and 
that $\rho$ modulo $p$ is trivial. Because the pro-$p$ completion  $\pi_1(X)^p$ of  the \'etale fundamental group $\pi_1(X)$ is 
isomorphic to a free pro-$p$ group $\langle \gamma_1, ..., \gamma_r \rangle$ of rank $r$, 
the composition 
$$  \pi_1(X) \to \pi_1(X)^p \cong \langle \gamma_1, ..., \gamma_r \rangle \to GL_2(\Z_p) $$ 
defined by $$\gamma_1 \mapsto \begin{pmatrix}  1 & p \\ 0 & 1 \end{pmatrix}, \,\,\,\,
\gamma_2 \mapsto \begin{pmatrix}  1 & 0 \\ p & 1 \end{pmatrix}, \,\,\,\,
\gamma_i \mapsto  \begin{pmatrix}  1 & 0 \\ 0 & 1 \end{pmatrix} \,\,(i \geq 3)$$
satisfies the conditions required for $\rho$. 
Let $(\sE,\Phi), \ \Phi:  F^*\sE \xrightarrow{\cong} \sE$, be the unit-root convergent $F$-isocrystal associated to 
$\rho_{\Q_p}: \pi_1(X) \to GL_2(\Z_p) \to GL_2(\Q_p)$ by 
\cite[Thm.~2.1]{Cre87} and let $E$ be the lattice of $\sE$ 
asociated to $\rho$ by the construction given in \cite[Prop.~2.3, Rmk.~2.3.2]{Cre87}.
Namely, if we denote the finite Galois covering associated to the kernel of 
$\rho$ modulo $p^n$ by $f_n: Y_n \to X$ and denote its Galois group by $G_n$, 
$E$ is defined to be the inverse limit of $(f_{n \,{\rm crys}*}\sO_{Y_n/W_n}^2)^{G_n}
\,(n \in \N)$, where the action of $G_n$ is induced by its action on $Y_n$ and $\rho$ modulo $p^n$. Then $E_X = (f_{1 \,{\rm crys}*}\sO_{Y_1/k}^2)^{G_1}$ is trivial because $\rho$ is trivial modulo $p$.

We prove that $\sE$ is irreducible as a convergent isocrystal. 
Assume on the contrary that $\sE$ has a convergent subisocrystal $\sE'$ of rank $1$. 
Then $0 \subsetneq \sE' \subsetneq \sE$ is a Jordan-H\"older sequence of $\sE$, 
and so is $0 \subsetneq F^*\sE' \subsetneq F^*\sE \cong \sE$. 
Hence $\{\sE', \sE/\sE'\} = \{ F^*\sE', F^*\sE/F^*\sE'\}$ and so 
we have an isomorphism $(F^*)^2\sE' \xrightarrow{\cong} \sE'$, which we denote by $\alpha$. 
By replacing $\alpha$ by $p^n \alpha$ for a suitable integer $n$, we may assume that 
the slope of $(\sE',\alpha)$ is $0$.
We denote the composite  $\Phi \circ F^*\Phi: (F^*)^2\sE \xrightarrow{\cong} \sE$ by 
$\Phi_2$. Let $\iota_1: \sE' \to \sE$ be the inclusion map and let $\iota_2: \sE' \to \sE$ be 
the composite $\sE' \xrightarrow{\alpha^{-1}} (F^*)^2\sE' \to (F^*)^2\sE \xrightarrow{\Phi_2} 
\sE$. 
Both $\iota_1, \iota_2$ are maps of convergent isocrystals. 
If $\iota_2$ is equal to $\iota_1$ up to scalar $a \in K \setminus \{0\}$, 
$(\sE',a \alpha)$ is a convergent $F$-subisocrystal of $(\sE,\Phi_2)$. Because 
 $(\sE,\Phi_2)$, corresponding to $\rho_{K_0}: \pi_1(X) \to GL_2(\Z_p) \to GL_2(K_0)$, 
has pure slope $0$, $(\sE',a \alpha)$ is necessarily of slope $0$ and so it induces 
a nontrivial subrepresentation of $\rho_{K_0}$, which contradicts the irreducibility 
of $\rho_{K_0}$. 
Hence $\iota_2$ is not a scalar multiple of $\iota_1$ and so 
$\iota := \iota_1 \oplus \iota_2$ defines an isomorphism $\sE' \oplus \sE' \xrightarrow{\cong} \sE$
of convergent isocrystals.
If we identify $\sE' \oplus \sE'$ and $\sE$ via $\iota$, the isomorphism 
$\Phi_2: (F^*)^2\sE' \oplus (F^*)^2\sE' \to \sE' \oplus \sE'$ 
is written in the form $f \circ (\alpha \oplus \alpha)$, where 
$f \in {\rm End}(\sE' \oplus \sE') \cong M(2 \times 2, K)$. 
Because $(\sE, \Phi_2)$, $(\sE',\alpha)$ have pure slope $0$, 
the $F$-isocrystal $(K^2,f)$ on $K$ has pure slope $0$. Hence 
there exists a nonzero $(x_1, x_2) \in K^2$ with $f(F^2(x_1),F^2(x_2)) = (x_1, x_2)$. 
Then we see that the image of $x_1 \oplus x_2: \sE' \to \sE' \oplus \sE' \cong \sE$ 
is stable under $\Phi_2$ and thus defines a convergent $F$-subisocrystal of $(\sE,\Phi_2)$, 
which again leads to a contradiction. Hence $\sE$ is irreducible as convergent isocrystal, 
as required.

Finally, the same argument as in Remark~\ref{rmk:irred} shows that there is 
no lattice $E'$ of $\sE$ such that 
$E'_X$ is $\mu$-stable in ${\rm Crys}(X/k)$.

\end{remark}

\begin{proof}[Proof of Theorem~\ref{thm:lefschetz}]
As $E_X$ is torsion free, we find a smooth complete intersection curve  $C$  in $ |\sH^n|$ such that $E_C$ as a coherent sheaf is locally free. 
Thus by Theorem~\ref{thm1:main}, $E_C$ has degree $0$. Thus $\mu(E_X)=0$, and for all smooth divisors 
$Y\in |\sH^n|$ with $E_Y$ torsion free, $\mu(E_Y)=0$. 
By \cite[Cor.~5.10]{Lan14} the Higgs sheaf $(V, \theta)$  associated  to $E_X$ by the Ogus-Vologodsky correspondence  \cite[Thm.~2.8]{OV07} is $\mu$-stable  of degree $0$
as well.  By \cite[Thm.~10]{Lan15}, $(V,\theta)|_Y$ is $\mu$-stable of degree $0$, thus again by  \cite[Cor.~5.10]{Lan14}, $E_Y$ is $\mu$-stable of degree $0$. When $\sE$ admits a locally free lattice, Theorem~\ref{thm1:main} and the independence 
of crystalline Chern classes with respect to lattices, the normalized Hilbert polynomial of $E_X$ is the same as 
that of $\sO_X$. Thus $E_X$ is locally free by \cite[Thm.~11]{Lan15} and so $E_Y$ is automatically torsion free. 
\end{proof}


\begin{thebibliography}{DK09-2}
 

\bibitem[Abe13]{Abe13}  T. Abe, {\it Langlands correspondence for isocrystals and existence of crystalline companions for curves},  {\tt http://arxiv.org/pdf/1310.0528.pdf}.
\bibitem[ABV05]{ABV05} 
F. Andreatta and L. Barbieri-Viale, {\it Crystalline realizations of 1-motives}, 
Math. Ann. {\bf 331}(2005),  no. 1, 111--172. 


%\bibitem[SGA4]{SGA4}  M. Artin, {\it Th\'eor\`eme de changement de base pour un morphisme lisse, et applications}, in Th\'eorie des Topos et Cohomologie \'Etale des Sch\'emas, SGA 4 XVI, Lecture Notes  in Mathematics {\bf 305}, Springer Verlag 1973.

\bibitem[Ber74]{Ber74}
P. Berthelot, {\it
Cohomologie cristalline des sch\'emas de caract\'eristique $p>0$, }
Lecture Notes in Math. {\bf 407}, Springer-Verlag, 1974. 

\bibitem[Ber96]{Ber96}
P. Berthelot, {\it
$\sD$-modules arithm\'etiques I: Op\'erateurs diff\'erentiels 
de niveau fini,} Ann. Sci. E.N.S. {\bf 29} (1996), 185--272.

%\bibitem[Ber00]{Ber00}
%P. Berthelot, {\it
%$\sD$-modules arithm\'etiques II: Descente par Frobenius,}
% M\'em. Soc. Math France (N.S.) {\bf 81} (2000), 1--136.

%\bibitem[Ber96b]{Ber96b} 
%P. Berthelot, {\it
%Cohomologie rigide et cohomologie rigide \`{a} supports propres
%\,\,\, premi\`{e}re partie, }pr\'{e}publication de l'IRMAR 96-03. 

\bibitem[Ber97]{Ber97} 
P. Berthelot, {\it Finitude 
et pueret\'e cohomologique en cohomologie rigide}, avec un appendice de A. J. de Jong,
Inventiones Math. {\bf 128} (1997), 329--377.

\bibitem[BI70]{BI70}
P. Berthelot and L. Illusie, 
{\it Classes de Chern en cohomologie cristalline}, 
Comptes Rendus Acad. Sciences Paris {\bf 270}(1970), 
1695--1697 and 1750--1752.

\bibitem[BO78]{BO78} 
P. Berthelot and A. Ogus, 
{\it Notes on crystalline cohomology, }
Princeton University Press, 1978. 

\bibitem[BO83]{BO83}
P. Berthelot and A. Ogus, 
{\it F-Isocrystals and De Rham Cohomology. I,}
Invent. math. {\bf 72} (1983), 159--199.


\bibitem[Che52]{Che52}
S.-S. Chern, {\it Differential geometry of fiber bundles},  Proceedings of the International Congress of Mathematicians, Cambridge, Mass.,  1950,  {\bf 2} , 397--411, Amer. Math. Soc., Providence, R. I. (1952).

\bibitem[CLS98]{CLS98}
B. Chiarellotto and B. Le Stum, 
{\it Sur la puret\'e de la cohomologie cristalline}, 
C. R. Acad. Sci. Paris {\bf 326}(1998), 961--963. 

\bibitem[ChT03]{ChT03}
B. Chiarellotto and N. Tsuzuki, 
{\it Cohomological descent of rigid cohomology for \'etale coverings}, 
Rend. del Sem. Mat. Univ. Padova {\bf 109} (2003), 63--215. 

\bibitem[Cre87]{Cre87} 
R. Crew, 
{\it $F$-isocrystals and $p$-adic representations}, Algebraic geometry, Bowdoin, 1985 (Brunswick, Maine, 1985),  Proc. Sympos. Pure Math., {\bf 46}, Part 2, 111--138, 
Amer. Math. Soc., Providence, RI (1987). 

%\bibitem[dJ97]{dJ97}
%J. de Jong,  {\it Smoothness, semi-stability and alterations},  Publ. Math.I/ H. \'E. S. {\bf 83} (1996), 51--93.


\bibitem[Del74]{Del74} P. Deligne, 
{\it Th\'eorie de Hodge: III},   Publ. Math. I. H. \'E. S. {\bf 44} (1984), 5--77. 




%\bibitem[Del80]{Del80} P. Deligne, 
%{\it La conjecture de Weil II,} Publ. Math. I. H. \'E. S. {\bf 52}  (1980), 137--252. 





\bibitem[Esn88]{Esn88} H. Esnault, {\it Characteristic classes of flat bundles}, Topology {\bf 27} 3 (1988), 323--352.

\bibitem[Esn92]{Esn92} H. Esnault, {\it Characteristic classes of flat bundles, II},  K-Theory {\bf  6} (1992), 45--56.

\bibitem[Esn00]{Esn00} H. Esnault, {\it Algebraic Differential Characters,}  in Regulators in Analysis, Geometry and Number Theory, Progress in Mathematics, Birkh\"auser Verlag, {\bf 171}  (2000), 89--117. 


%\bibitem[EM10]{EM10} 
%H. Esnault and V. Mehta, {\it 
%Simply connected projective manifolds
%in characteristic $p > 0$ have no nontrivial stratified
%bundles,} Invent. Math. {\bf 181} (2010), 449--465. 
%(Erratum available at {\tt http://www.mi.fu-berlin.de/users/esnault/helene\_publ.html, 95b }
%(2013)). 


\bibitem[ES15]{ES15}  H. Esnault and A. Shiho, {\it Convergent isocrystals on simply connected varieties}, 
{\tt http://www.mi.fu-berlin.de/users/esnault/helene\_publ.html, 116 }.


\bibitem[Fri82]{Fri82} Friedlander, E.: {\it \'Etale homotopy of simplicial schemes}, Princeton University Press (1982). 

%\bibitem[Gie75]{Gie75}
%D. Gieseker, 
%{\it Flat vector bundles and the fundamental group in non-zero characteristics,} Ann. Sc. Norm. Super. Pisa, 4 S\'er. {\bf 2} (1) (1975), 1--31. 

% \bibitem[SGA2]{SGA2} A. Grothendieck, {\it ...}

\bibitem[Gro85]{Gro85}
M. Gros, {\it 
Classes de Chern et classes de cycles en cohomologie 
de Hodge-Witt logarithmique, }
M\'em. Soc. Math. France {\bf 21} (1985), 1--87. 

\bibitem[Gro90]{Gro90}
M. Gros, {\it 
R\'egulateurs syntomiques et valeurs de fonctions $L$  $p$-adiques I}, 
% With an appendix by Masato Kurihara. 
Invent. Math.  {\bf 99}(1990), 293--320. 

%\bibitem[EGAIV4]{EGAIV4} A. Grothendieck, 
%{\it \'Etude locale des sch\'emas et des morphismes de sch\'emas}, Publ. Math. I. H. \'E. S. {\bf 32} (1997).

%\bibitem[Gro68]{Gro68}
%A. Grothendieck,
%{\it Crystals and de Rham cohomology of schemes}, 
%in Dix Expos\'es sur la Cohomologie des Sch\'emas, North Holland Publishing Company 1968,  306--358.


%\bibitem[Gro70]{Gro70} 
%A. Grothendieck,  {\it
%Repr\'esentations lin\'eaires et 
%compactifications profinies des groupes discrets, }
%Manuscr. Math. {\bf 2} (1970), 375--396. 

%\bibitem[SGA1]{SGA1} 
%A. Grothendieck,
%{\it  Rev\^etements \'etales et groupe fondamental,} 
%Lecture Notes in Math.  {\bf 224}, Springer Verlag 1971.

%\bibitem[HL97]{HL97}
%D. Huybrechts and M. Lehn, {\it
%The geometry of moduli spaces of sheaves, }
%Aspects of Mathematics, vol. {\bf E31}. Vieweg, Braunschweig (1997). 



\bibitem[Ill79]{Ill79} L. Illusie .:  {\it Complexe de de Rham-Witt et cohomologie cristalline}, Publ. Sc. \'E.N.S. 
{\bf 12} 4i\`eme s\'erie, 501--661.

\bibitem[Ked04]{Ked04}
K. S. Kedlaya, {\it Full faithfulness for overconvergent $F$-isocrystals}, 
Geometric aspects of Dwork theory 819--835, Walter de Gruyter, Berlin, 2004. 

\bibitem[KM74]{KM74} N. Katz and W. Messing, {\it Some consequences of the Riemann hypothesis for varieties over finite fields},  Invent. math. {\bf 23} (1974), 73--77.



%\bibitem[LSZ13]{LSZ13}
%G. Lan, M. Sheng and K. Zuo, {\it
%Semistable Higgs bundles, periodic Higgs bundles and representations of algebraic fundamental groups,}
%{\tt  arXiv:1311.6424v2. }

%\bibitem[Lan04]{Lan04}
%A. Langer, 
%{\it Semistable sheaves in positive characteristic, }
%Ann. Math. {\bf 159} (2004), 251--276. 

%\bibitem[Lan04b]{Lan04b} 
%A. Langer, {\it 
%Moduli spaces of sheaves in mixed characteristics, }
%Duke Math. J. {\bf 124}(3) (2004),
%571--586. 

%\bibitem[Lan11]{Lan11}
%A. Langer, {\it
%On the $S$-fundamental group scheme, }
%Ann. Inst. Fourier (Grenoble) {\bf 61} (2011), 
%2077--2119.

\bibitem[Lan14]{Lan14}
A. Langer, {\it 
Semistable modules over Lie algebroids 
in positive characteristic, }
Doc. Math. {\bf 19} (2014), 509--540. 

\bibitem[Lan15]{Lan15} 
A. Langer, {\it
Bogomolov's inequality for Higgs sheaves
in positive characteristic, }
 Invent. Math., 

%\bibitem[Lan75]{Lan75} 
%S. G. Langton,  {\it
%Valuative criteria for families of vector bundles on algebraic varieties, }
%Ann. of Math., {\bf 101} (1975), 88--110. 


\bibitem[Laz15]{Laz15}
C. Lazda, {\it 
Incarnations of Berthelot's conjecture}, 
{\tt http://arxiv.org/abs/1508.06787}

\bibitem[LeS07]{LeS07} 
B. Le Stum, {\it
Rigid cohomology, }
Cambridge Tracts in Mathematics {\bf 172}, 
Cambridge University Press, 2007. 

%\bibitem[Mal40]{Mal40}
%A. Mal\v{c}ev, {\it
%On isomorphic matrix representations of infinite groups, }
%Mat. Sb. N.S. {\bf 8}(50) (1940), 405--422. 

\bibitem[Man13]{Man13} 
P. Mannisto, {\it Albanese and Picard $1$-motives in positive 
characteristic},  
{\tt arXiv:1308.0472}.

\bibitem[Mar11]{Mar11} N. Markaryan, {\it  Crystalline characteristic classes via the Atiyah gerbe and generalized Chern-Weil construction,} notes, 12 pages, not posted.


%\bibitem[MR83]{MR83}
%V. B. Mehta and A. Ramanathan, {\it
%Homogeneous bundles in characteristic $p$,}  pp. 315--320 in 
%Algebraic geometry---open problems (Ravello, 1982), Lecture Notes in Math. {\bf 997}, 
%Springer  Verlag  1983. 

%\bibitem[Mum70]{Mum70}
%D. Mumford, {\it
%Abelian varieties, } Oxford University Press, London , 1970. 

%\bibitem[Ogu79]{Ogu79} Ogus, A.: {\it Supersingular $K3$ crystals},  Journ\'ees de G\'eom\'etrie Alg\'ebrique de Rennes, Ast\'erisque {\bf 64}, (1979), 3--86.

%\bibitem[Ogu83]{Ogu83} Ogus, A.: {\it A crystalline Torelli theorem for supersingular $K3$ surfaces}, Progr. Math. {\bf 36} (1983) 361--394. 




\bibitem[Ogu84]{Ogu84}
A. Ogus, {\it
$F$-isocrystals and de Rham cohomology II  --- 
Convergent isocrystals,}
 Duke Math. J., {\bf 51} (1984), 765--850.

%\bibitem[Ogu90]{Ogu90}
%A. Ogus, 
%{\it The convergent topos in characteristic $p$,}
% pp. 133--162 
%in Grothendieck Festschrift, Progress in Math. {\bf 88}, 
%Birkh\"{a}user (1990). 

\bibitem[Ogu94]{Ogu94} 
A. Ogus, 
{\it F -crystals, Griffiths transversality, and the Hodge decomposition}, 
Ast\'erisque {\bf 221}(1994). 

\bibitem[OV07]{OV07}
A. Ogus and V. Vologodsky, {\it
Nonabelian Hodge theory in characteristic $p$, }
Publ. Math. IHES {\bf 106} (2007), 1--138. 

%\bibitem[Pet03]{Pet03} 
%D. Petrequin, 
%{\it Classes de Chern et classes de cycles en cohomology rigide}, Bull. Soc. math. France, {\bf 131} (1) (2003), 59--121. 

%\bibitem[RR84]{RR84}
%S. Ramanan and A. Ramanathan, 
%{\it Some remarks on the instability flag, }
%Tohoku Math. J. {\bf 36} (1984), 269--291. 

%\bibitem[Shi14]{Shi14}
%A. Shiho, {\it A note on convergent isocrystals on simply connected varieties,} 
%{\tt arXiv:1411.0456}.


\bibitem[Tat66]{Tat66} J. Tate,
{\it Endomorphisms of abelian varieties over finite fields},  Invent. math. {\bf 2} (1966), 134--144. 



%\bibitem[Tsu12]{Tsu12}
%N. Tsuzuki, {\it A Note on the First Rigid Cohomology Group
%for Geometrically Unibranch Varieties}, Rend. Sem. Mat. Univ. Padova, {\bf 128} (2012), 17--53.


\end{thebibliography}
\end{document}